\documentclass[11pt,a4paper]{article}

\usepackage[english]{babel}
\usepackage{fancyhdr}
\usepackage[latin1]{inputenc}
\usepackage{amsmath}
\usepackage{amsfonts}
\usepackage{amssymb}
\usepackage{graphicx}
\usepackage{latexsym}
\usepackage{amsthm}
\usepackage{color}
\usepackage{midpage}
\usepackage[width=15.00cm]{geometry}
\usepackage{amscd}
\usepackage{tikz-cd}
\usepackage{hyperref}

\numberwithin{equation}{section}
\theoremstyle{plain}

\theoremstyle{definition}
\newcommand{\eps}{\varepsilon}

\numberwithin{equation}{section}
\theoremstyle{plain}
\newtheorem{thm}{Theorem}[section]
\theoremstyle{remark}

\theoremstyle{cor}
\newtheorem{cor}[thm]{Corollary}
\theoremstyle{lem}
\newtheorem{lem}[thm]{Lemma}
\theoremstyle{proposition}
\newtheorem{prop}[thm]{Proposition}
\DeclareMathOperator{\len}{len}

\DeclareMathOperator{\Cov}{Cov}
\DeclareMathOperator{\Vol}{Vol}
\DeclareMathOperator{\Corr}{Corr}
\DeclareMathOperator{\Var}{Var}
\providecommand{\keywords}[1]{\textbf{Keywords and Phrases:} #1}
\providecommand{\classification}[1]{\textbf{AMS Classification:}#1}

\begin{document}

\title{Nodal Lengths in Shrinking Domains for Random Eigenfunctions on $S^2$ }

\author{{Anna Paola Todino}\\Ruhr-Universit\"{a}t Bochum
	\\ \small Anna.Todino@ruhr-uni-bochum.de}

 \date{}
 \title{ Nodal Lengths in Shrinking Domains for\\ Random Eigenfunctions on $S^2$}
 \maketitle




\begin{abstract}
We investigate the asymptotic behavior of the nodal lines for random spherical harmonics restricted to shrinking domains, in the 2-dimensional case: e.g., the length of the zero set $\mathcal{Z}_{\ell,r_\ell} := \mathcal{Z}^{B_{r_{\ell}}}(T_\ell)=\len(\{x \in S^2 \cap B_{r_\ell}: T_\ell(x)=0 \})$, where $B_{r_{\ell}}$ is the spherical cap of radius $r_\ell$. We show that the variance of the nodal length is logarithmic in the high energy limit; moreover, it is asymptotically fully equivalent, in the $L^2$-sense, to the ``local sample trispectrum", namely, the integral on the ball of the fourth-order Hermite polynomial. This result extends and generalizes some recent findings for the full spherical case. As a consequence a Central Limit Theorem is established.
\end{abstract}

\begin{itemize}
	\item \keywords{} Random Eigenfunctions, Limit Theorem, Sample Trispectrum, Berry's Cancellation.
	
	\item \classification{} 35P20, 60F05, 58J50, 60G60.
\end{itemize}


	\section{Introduction and Background}
Let us consider the spherical Laplacian $\Delta_{S^2}$, defined as usual by $$\Delta_{S^2}=\frac{1}{\sin \theta} \frac{\partial}{\partial \theta} \bigg\{ \sin \theta \frac{\partial}{\partial \theta} \bigg\}+ \frac{1}{\sin^2 \theta} \frac{\partial}{\partial^2 \varphi}$$ and $\{T_\ell(x), x \in S^2\}$, satisfying
		 $\Delta_{S^2} T_{\ell}(x)+ \ell(\ell+1)T_{\ell}(x)=0,$ the centred isotropic Gaussian random spherical harmonics with covariance function given by 
		$$E[T_\ell(x) T_\ell(y)]=P_\ell(\cos d(x,y)),$$ being $P_\ell$ the Legendre polynomial and $d(x,y)$ the spherical geodesic distance between $x$ and $y$, $d(x,y)=\arccos (\langle x,y \rangle)$.
As usual, the nodal set of $T_\ell$ is given by $T_\ell^{-1}(0)=\{ x \in {S}^2: T_\ell(x)=0  \}$ and we denote its volume by
	\begin{equation}\label{total sphere}
	\mathcal{Z}(T_\ell)=\len(\{ x \in {S}^2 : T_\ell(x)=0 \});
	\end{equation}
	the analysis of these domains has been considered by many authors, see e.g.  \cite{Cheng}, \cite{Y1982}, \cite{Y1990}, \cite{DF}, \cite{BR11}, \cite{BW2016}. As a consequence of the general Yau's conjecture (\cite{Y1982}, \cite{Y1990}) for eigenfunctions on compact manifolds (proved in \cite{DF} for real analytic metrics and by \cite{Logunov a}, \cite{Logunov b} and \cite{LM} for the smooth case) we know that, in the high energy limit, the length of the nodal set is bounded by
	$$ c_1 \sqrt{\ell(\ell+1)} \leq \len(T_\ell^{-1}(0))\leq c_2 \sqrt{\ell(\ell+1)},$$ where $c_1, c_2 > 0$.
	In the case of Gaussian random eigenfunctions, some sharper probabilistic bounds can be given. The asymptotic behavior of the expected value was given in \cite{berard}; for any dimension $m, m \geq 2$, they obtained
	$${E}[\mathcal{Z}(T_\ell)^m]=c_m \sqrt{\ell(\ell+m-1)},$$ where $c_m=\dfrac{2\pi^{m/2}}{\sqrt{m} \Gamma(\frac{m}{2})}$ (see also \cite{N} and \cite{35}).
	As far as the variance is concerned, \cite{N} gave an upper bound which was later improved in \cite{35} and \cite{W}, where it was computed to be
		\begin{equation}\label{varianza vecchia}
	\Var(\mathcal{Z}(T_\ell))=\dfrac{1}{32} \log {\ell} +O(1)
	\end{equation}
	as $\ell \rightarrow \infty$. As a consequence, the variance of the nodal length $\mathcal{Z}(T_\ell)$ has smaller order $O(\log \ell)$, in the high energy limit, with respect to the variance of boundary length at thresholds different from zero, which has been shown to be $O(\ell)$ (see for instance \cite{Rossi}). This phenomenon is known as ``Berry's cancellation" (\cite{Berry1977}); it is known to occur on the torus (\cite{KKW}) and on other geometric functionals of random eigenfunctions, see e.g., \cite{CMW}, \cite{CMW2016}, \cite{CM}. More precisely, as far as the torus is concerned, \cite{RW} and \cite{KKW} studied the volume of the nodal line (denoted with $\mathcal{L}_\ell$) of random eigenfunctions (``arithmetic random waves") $\mathcal{T}^2=R^2/Z^2$. The expected length was evaluated with the Kac-Rice formula in \cite{RW} (Proposition 4.1),
	$$E[\mathcal{L}_\ell]= \dfrac{1}{2\sqrt{2}} \sqrt{4\pi^2 \ell},$$
	and the asymptotic behavior of the variance was established in \cite{KKW}; it holds that
	$$\Var(\mathcal{L}_\ell)=c_\ell \cdot \dfrac{4\pi^2 \ell}{\mathcal{N}_\ell^2} \bigg(1+O\bigg(\dfrac{1}{\mathcal{N}_\ell^{1/2}}\bigg)\bigg),$$
	where $\mathcal{N}_\ell$ is the number of lattice points lying on the radius-$\sqrt{\ell}$ circle (\cite{KKW}) and $c_\ell$ is the leading coefficient, depending on the distribution of the lattice points on the circle.
	Hence, as mentioned before, the ``Berry's cancellation" phenomenon (\cite{Berry1977}) takes place also for the toral nodal length. The distribution of $\mathcal{L}_\ell$ was investigated in \cite{MPRW}, where the authors established a nonCentral Limit Theorem. See also \cite{RossiW} for nodal intersections, \cite{BW2016} for the number of nodal domains.
	Berry's random planar wave model was also considered (see \cite{NPR}), both in the real and complex case.\\\\
	A general interpretation of these results can be given quickly as follows (see \cite{MPRW}, \cite{MRW}, \cite{Rossi2} for more discussions and details). The nodal length $\mathcal{L}_\ell$ of random eigenfunctions can be expanded, in the $L^2-$sense, in terms of its $q$-th order chaotic components, to obtain the orthogonal expansion:
	$$\mathcal{L}_\ell-E[\mathcal{L}_\ell]=\sum_{q=1}^{\infty} Proj[\mathcal{L}_\ell|C_q],$$
	$Proj[\mathcal{L}_\ell|C_q]$ denoting the projection on the $q-$component (see the supplement article \cite{supp}, Section A.1). It can be shown that, in the case of functionals evaluated on the full sphere or torus, the projection on the first component vanishes identically; in the nodal case, $Proj[\mathcal{L}_\ell|C_2]$ vanishes as well, and the whole series is dominated simply by the term $Proj[\mathcal{L}_\ell|C_4]$, e.g., the so-called fourth-order chaos, which has indeed logarithmic variance. More explicitly, the variance of this single term is asymptotically equivalent to the variance of the full series, and its asymptotic distribution (Gaussian in the spherical case, nonGaussian for the torus, see \cite{RW}) gives also the limiting behavior of the nodal fluctuations.
	It should also be noted that, in the case of the sphere, $Proj[\mathcal{L}_\ell|C_4]$ takes a very simple form, because it is proportional to the so-called sample trispectrum of $T_\ell$, $\int_{S^2} H_4(T_\ell(x))\,dx$ (being $H_j$ the $j-$th Hermite polynomial): this is to some extent unexpected, because the fourth-order chaotic term should in general be given by a complicated linear combination of polynomials involving also the gradient of the eigenfunctions (see the supplement article \cite{supp}, Section A.1.1), as it happens for arithmetic random waves on the torus, see \cite{MPRW}).\\\\
	A natural question at this stage is to investigate what happens on subdomains of the sphere or other manifolds (see for example \cite{BMW} for arithmetic random waves). The nodal volume inside a ``nice" domain $F \subset S^2$ of the sphere, is defined as
	\begin{equation}\label{alessio}
	\mathcal{Z}^F(T_\ell):= \len(\{T_\ell=0 \}\cap F).
	\end{equation}
	In \cite{W}, to address this issue the so-called \textit{linear statistics} of the nodal set are introduced; more precisely, let $\varphi:S^2 \rightarrow R$ be a smooth function, and define the random variable $\mathcal{Z}^\varphi(T_\ell)$ as
	\begin{equation}\label{linear stat}
	\mathcal{Z}^\varphi(T_\ell):= \int_{T_\ell^{-1}(0)} \varphi(x) \,d\len_{T_\ell^{-1}(0)}(x).
	\end{equation}
	Apparently this definition is well-posed only for continuous test function $\varphi \in C(S^2)$; nevertheless, it was shown in \cite{W} that bounded variation functions $BV({S}^2)$ can be considered: indeed, it is possible to prove that, for $\varphi \in BV({S}^2) \cap L^{\infty}({S}^2)$ a not identically vanishing function, as $\ell \rightarrow \infty$, the variance satisfies
	\begin{equation}\label{V1}
	\Var(\mathcal{Z}^\varphi(T_\ell))=\dfrac{||\varphi||_{L^2({S}^2)}^2}{128 \pi} \cdot \log {\ell} +O_{\varphi}(1).
	\end{equation}
	These results allow to cover indicator functions, indeed (\ref{alessio}) is equal to (\ref{linear stat}) for $\varphi(x)=1_F(x)$, e.g. $\mathcal{Z}^{\varphi}(T_\ell)=\len (\{x \in {S}^2 \cap F: T_\ell(x)=0 \}).$

	As a consequence of (\ref{V1}), for $F \subset {S}^2$ a submanifold of the sphere with $C^2$ boundary, and $|F|$ denotes its area, it was proved in \cite{W} that, as $\ell \rightarrow \infty$, the variance of (\ref{alessio}) is given by:
	$$\Var(\mathcal{Z}^F(T_\ell))=\dfrac{|F|}{128\pi} \cdot \log {\ell}+O_F(1),$$
	e.g., logarithmic behavior occurs also in subdomains.
	
	As far as the torus is concerned, the nodal length of arithmetic random waves restricted to shrinking balls (denoted with $\mathcal{L}_{\ell,r_{\ell}}$, where $r_{\ell}$ is the radius of the ball) was investigated in \cite{BMW} under the condition $r_\ell > \ell^{-1/2}$. The mean was easily obtained by means of Kac-Rice formula (\cite{Adler e Taylor}, \cite{Azais e ws})
	$${E}[\mathcal{L}_{\ell,r_{\ell}}]=\dfrac{1}{2\sqrt{2}}(\pi r_\ell^2)\cdot \sqrt{4\pi^2\ell}, $$
	whereas the variance was shown to be proportional to the variance of the toral nodal length, e.g.,
	$$\Var(\mathcal{L}_{\ell,r_{\ell}})=c_\ell \cdot (\pi r_\ell^2)^2 \cdot \dfrac{4\pi^2 \ell}{\mathcal{N}_\ell^2} \bigg(1+O\bigg(\dfrac{1}{\mathcal{N}_\ell^{1/2}}\bigg)\bigg).$$
	More surprisingly, it was shown that asymptotically the local and global nodal lengths are fully correlated. This result entails also that, up to a scaling factor, the same limiting nonGaussian distribution holds in both cases. 
	
	\section{Main Results}
	In this paper, we investigate the behavior of the nodal length for random spherical harmonics evaluated in a shrinking ball on the sphere. Without loss of generality, we consider spherical caps centered in the North Pole $N$. We prove that the nodal length is still dominated by a single term, corresponding to the fourth chaotic projection; moreover, this term can be written as a local form of the sample trispectrum, and its asymptotic variance is logarithmic (e.g., $O(r_\ell^2 \log(r_\ell \ell ))$). Contrary to the case of the torus, however, full correlation does not hold between nodal and global statistics. ``Berry's cancellation" phenomenon takes place in this framework as well, and indeed the first and second order chaotic components are still of lower order with respect to the leading term, although not identically equal to zero as in the full spherical case.\\\\
	Here and in the rest of the paper we will always denote with $B_{r_{\ell}} \subset {S}^2$ a shrinking spherical cap of radius $r_{\ell}$, with $r_\ell \rightarrow 0$, as $\ell \rightarrow \infty$, centered in $N$ such that
	\begin{equation}\label{cond su r}
	r_{\ell}\ell \rightarrow \infty
	\end{equation}
	as $\ell \rightarrow \infty$ (meaning that the support is not shrinking too rapidly). Indeed, the average length on the disc of radius $r_\ell$ is $r_\ell^2 \ell <r_\ell \ell$; hence, if condition (\ref{cond su r}) is not satisfied, we cannot expect any asymptotic result to observe.
	We denote the nodal length in these domains by
	\begin{equation}\label{def}
	\mathcal{Z}_{\ell,r_\ell} := \mathcal{Z}^{B_{r_\ell}}(T_\ell)=\len(\{x \in {S}^2 \cap B_{r_\ell}: T_\ell(x)=0 \}).
	\end{equation}
	From the Kac-Rice formula (\cite{Adler e Taylor}, \cite{Azais e ws}), it is easy to see that
	$${E}[\mathcal{Z}_{\ell,r_\ell}]=\sqrt{\dfrac{\ell(\ell+1)}{2}} \dfrac{|B_{r_{\ell}}|}{2}.$$
	Note that, since the area of a spherical cap $B_{r_\ell}$ of radius $r_\ell$ is given by $|B_{r_\ell}|=2\pi(1-\cos r_\ell),$ we have that
	$${E}[\mathcal{Z}_{\ell,r_\ell}]=\sqrt{\dfrac{\ell(\ell+1)}{2}} \pi (1-\cos r_\ell).$$
	Now let $\varphi_{\ell}: {S}^2 \rightarrow {R}$, $\forall \ell $, be the indicator function $\varphi_{\ell}(x)=1_{B_{r_{\ell}}}(x)$; our first non-trivial result concerns the asymptotic variance is the following.
	\begin{thm}\label{varianza nodal} Let $\mathcal{Z}_{\ell,r_\ell}$ be the nodal length defined in (\ref{def}), then its variance, as $\ell \rightarrow \infty,$ is given by
		\begin{equation}\label{variance}
		\Var(\mathcal{Z}_{\ell,r_{\ell}})=\dfrac{1}{256} \cdot r_{\ell}^2 \log(r_{\ell} \ell)+O(r_{\ell}^2).
		\end{equation}
	\end{thm}
	The next result is the following Central Limit Theorem.
	
	\begin{thm}\label{CLT}
		Let $\mathcal{Z}_{\ell,r_\ell}$ defined in (\ref{def}), then, as $\ell \rightarrow \infty$, we have that
		$$\dfrac{\mathcal{Z}_{\ell,r_\ell}-{E}[\mathcal{Z}_{\ell,r_\ell}]}{\sqrt{\Var(\mathcal{Z}_{\ell,r_\ell})}} \rightarrow_{d} Z,$$
		where $\rightarrow_{d}$ denote the convergence in distribution and $Z \sim \mathcal{N}(0,1)$.
	\end{thm}
	
	Theorem \ref{CLT} follows by exploiting Theorem 5.2.6 in \cite{Peccati Nourdin} to the fourth chaotic component, after lengthy computations of the fourth cumulant (which is, for $Y$ a centred random variable, $cum_4(Y)=EY^4-3(EY^2)^2$) of this chaotic projection.
	
	\subsection{Comparison with the 2-dimensional Torus}
	Although the differences and the similarities of the results obtained for the torus and for the sphere have already been discussed, we make them clearer in this subsection.
	\begin{itemize}
		\item In contrast to the torus, where a full correlation between the nodal length in shrinking domains and the one in the total manifold has been proved (see \cite{BMW}), in the sphere the following proposition holds.
		\begin{prop}\label{correlation1} Let $\mathcal{Z}_{\ell,r_{\ell}}$ be defined in (\ref{def}) and $\mathcal{Z}(T_\ell) $ in (\ref{total sphere}), the correlation between $\mathcal{Z}_{\ell,r_{\ell}}$ and $\mathcal{Z}(T_\ell) $, as $\ell \rightarrow \infty$, is given by
			$$\Corr(\mathcal{Z}_{\ell,r_{\ell}}; \mathcal{Z}(T_\ell))= O\bigg(r_{\ell} \sqrt{\dfrac{\log \ell}{\log r_\ell \ell}}\bigg)=o(1).$$
		\end{prop}
		Proposition \ref{correlation1} entails on the contrary that the correlation between the ``local" and ``global" nodal length is zero, in the high frequency limit. The discrepancy between these two results can be heuristically explained as follows: in the case of the torus, local integrals for products of four eigenfunctions have the same form, whatever the centre of the disc on which they are computed (see \cite{BMW}). This is not the case when integral of the products of four spherical harmonics is computed on a disc; this integral has different values depending on the centre of the disc and because of this full correlation cannot be expected. 
		\item In the case of the torus, the full correlation result allows to establish immediately the (nonCentral) Limit Theorem for the nodal length in the shrinking set; indeed, the ``local" limiting distribution is the same as the ``global" one, up to a different scaling constant. On the contrary, to establish a (Central) Limit Theorem for the spherical cap, a different proof is required; indeed we need to apply Theorem 5.2.6 in \cite{Peccati Nourdin} and hence to compute the fourth cumulant of the leading chaos projection of the nodal length.
		In passing we stress that the limiting in distribution is Gaussian in the present framework, while it is a linear combinations of Chi-square random variables in the torus.
		\item In both the manifolds and their subregions, the fourth chaotic component is the leading term of the chaos expansion of the nodal length and the ``Berry's cancellation" phenomenon occurs. However, only in the sphere and in its subdomains, the dominant component is asymptotic to the sample trispectrum, e.g. it has a much simpler form as the integral of the fourth Hermite polynomial, computed only on the eigenfunctions themselves.
	\end{itemize}
	
	\subsection{Plan of the paper}
In Section \ref{Section-ontheproof} we explain the basic ideas for proving the main results of the paper; while the main tools to succeed in our computations are introduced in Section \ref{Section-auxiliary}, where an auxiliary function and its properties and the construction of a smooth approximation of the indicator function are discussed. Chapter \ref{section-proofs} is splitted in two subsections; \ref{section-proofvar} contains the proof of the asymptotic behavior of the variance and \ref{section-proofCLT} proves the Central Limit Theorem. In Section \ref{Section-further} the correlation between ``local" and ``global" nodal length is computed and finally Section \ref{App2point} collects some technical tools exploited in the computations.
	
	\subsection{Some conventions}
	Given a set $F \subset {S}^2,$ we denote its area by $|F|$ and for a smooth curve $C \subset {S}^2, $ $\len(C)$ its length. We will use $A\ll B$ and $A=O(B)$ in the same way. $O_\varphi$ means that the constants involved depend on the function $\varphi$ and they stay bounded when $\varphi$ stays bounded.
	
	\section{On the proof of the main results}\label{Section-ontheproof}
	In this section we give the guideline of the proof of the main results. In the full sphere, it is possible to write the second moment as
	\begin{equation}\label{1}
E[(\mathcal{Z}(T_\ell))^2]= \int_{S^2\times {S}^2} \tilde{K}_{\ell}(x,y) \,dx dy
	\end{equation}
	(see \cite{BSZ} Theorem 2.2, \cite{BSZ2} Theorem 4.3, \cite{35} Proposition 3.3),
	where
	$\tilde{K}_{\ell}(x,y)=\tilde{K}_{\ell}(d(x,y))$ is the two-point correlation function (see Section \ref{App2point}), and the symmetry of the domain implies that, changing coordinates, (\ref{1}) yields
	$$E[(\mathcal{Z}(T_\ell))^2]= 8\pi^2 \int_{0}^{\pi} \tilde{K}_{\ell}(\rho) \sin \rho \,d\rho$$ which allows to handle the computations and to establish the asymptotic behavior of the variance.
	Focussing instead on a subdomain, the lack of this symmetry prevents this change of coordinates. However, using (\ref{linear stat}) and the same argument as in \cite{W} (Proof of Theorem 1.4), it can be shown that for any function $\varphi: {S}^2 \rightarrow {R}$ in $C^1({S}^2)$, we have that
	$${E}[(\mathcal{Z}^{\varphi}(T_\ell))^2]= \int_{{S}^2\times {S}^2} \varphi(x) \varphi(y) \tilde{K}_{\ell}(x,y) \,dx dy.$$
	Now, introducing an auxiliary function $W^{\varphi}:[0,\pi] \rightarrow {R}$ (see also \cite{W}), defined as
	\begin{equation}\label{Wde}
	W^{\varphi}(\rho):=\dfrac{1}{8\pi^2} \int_{d(x,y)=\rho} \varphi(x) \varphi(y) \,dx \,dy \mbox{ } \mbox{ } \mbox{ } x,y \in  {S}^2,
	\end{equation}
	and employing Fubini, we get that
	$${E}[(\mathcal{Z}^{\varphi}(T_\ell))^2]= 8\pi^2 \int_{0}^{\pi} \tilde{K}_{\ell}(\rho) W^{\varphi}(\rho) \,d\rho$$
	with $\tilde{K}_{\ell}(\rho)=\tilde{K}_{\ell}(x,y),$
	$x,y \in {S}^2$ being any pair of points with $d(x,y)=\rho.$
	The crucial observation is that the case of a spherical cap can be cast in this framework, simply taking $\varphi=1_{B_{r_{\ell}}}$, which is a function in $ BV({S}^2) \cap L^\infty({S}^2)$, $\forall \ell$.\\
	More precisely, the key role in the proof of Theorem \ref{varianza nodal} will be played by a sequence of auxiliary functions, $W^{\varphi_\ell}:[0,2r_\ell] \rightarrow {R}$, defined as
	\begin{equation}\label{Wdef}
	W^{\varphi_\ell}(\rho):=\dfrac{1}{8\pi^2} \int_{d(x,y)=\rho} \varphi_\ell(x) \varphi_\ell(y) \,dx \,dy \mbox{ } \mbox{ } \mbox{ } x,y \in  {S}^2;
	\end{equation}
	and using a density argument and approximating $1_{B_{r_\ell}}$ with $C^1$ functions $\varphi_\ell^i$, the second moment could be written as
	$${E}[(\mathcal{Z}^{\varphi^{i}_{\ell}}(T_\ell))^2]= 8\pi^2 \int_{0}^{2r_\ell} \tilde{K}_{\ell}(\rho) W^{\varphi_{\ell}^i}(\rho) \,d\rho.$$ 
	Note that (\ref{Wdef}) is not zero if and only if the variables $x,y $ are inside the spherical cap $B_{r_\ell}$, hence the maximum distance allowed between two points to make (\ref{Wdef}) different from zero is $\rho=2r_{\ell}$.
	For $\varphi_{\ell}=1_{B_{r_{\ell}}}$ and for $x,y \in B_{r_\ell}$, (\ref{Wdef}) can be written also as
	$$W^{\varphi_{\ell}}(\rho)= \dfrac{1}{8\pi^2} \int_{B_{r_{\ell}}} \len\{ y \in B_{r_{\ell}}: d(x,y)=\rho \} \,dx.$$ Then, if we fix $x$ ``far" from the boundary, the integrand will be given by $\len \{y \in B_{r_\ell}: d(x,y)=\rho\} =2\pi \sin \rho;$ note that,
	however, $W^{\varphi_{\ell}}$ depends on the position of $x$. Moreover,  for decreasing sequence $r_{\ell}$ a tangent plane approximation can be shown to hold, whence, we can also define the function
	$\tilde{W}_{\tilde{\varphi}_\ell}: [0,2r_\ell] \rightarrow {R}$ as
	\begin{equation}\label{Wtilde}
	\tilde{W}_{\tilde{\varphi_\ell}}(\rho):=\dfrac{1}{8\pi^2}\int_{d(x,y)=\rho} \tilde{\varphi_\ell}(x) \tilde{\varphi_\ell}(y) dx dy \mbox{ } \mbox{ } \mbox{ } x,y \in {R}^2,
	\end{equation}
	where $\tilde{\varphi}_\ell$ is given by the composition $\varphi_{\ell} \circ \exp$ and $\exp$ is the exponential map. Note that $\tilde{W}_{\tilde{\varphi}_{\ell}}$ is nonzero if $x,y \in \tilde{B}_{r_\ell}$, which is the disc contained in ${R}^2$ of radius $r_\ell$ and centered in the origin of the axes. In order to scale the support of $\tilde{\varphi}_\ell$ from $\tilde{B}_{r_\ell}$ in $\tilde{B}_1$, we define also
	\begin{equation}\label{W_1}
	\tilde{W_1} \bigg(\rho \dfrac{1}{r_\ell} \bigg):= \frac{1}{8\pi^2}\int_{d(x,y)=\frac{\rho}{r_\ell}}\tilde{\varphi_\ell}(r_\ell x) \tilde{\varphi_\ell}(r_\ell y) \,dx\,dy \mbox{ }\mbox{ }\mbox{ } x,y \in {R}^2.
	\end{equation}
	Denoting $W_{r_\ell}(\rho):=W^{1_{B_{r_\ell}}}(\rho)$ (e.g. $\varphi_\ell=1_{B_{r_{\ell}}}$), it is easy to check the validity of the asymptotic relation below:
	\begin{equation}\label{1111}
	W_{r_{\ell}}(\rho)=r_{\ell}^3 \tilde{W_{1}}\left(\rho \frac{1}{r_{\ell}}\right)(1+O(\rho^2)),
	\end{equation}
	as $r_{\ell} \rightarrow 0 $ uniformly in $\rho$ (see Lemma B.3 in the supplement article \cite{supp}).\\
	
	Hence, as we said before, in order to prove Theorem \ref{varianza nodal} we want to apply a standard approximation argument; approximating the characteristic function $1_{B_{r_\ell}}$ with a sequence of $C^1$ function for which we can apply the following Proposition \ref{var for cont func}.
			\begin{prop}\label{var for cont func} Let $\varphi_\ell^i$ be a sequence of $C^1$ functions satisfying (\ref{properties}) and let define $\tilde{\varphi}_\ell^i(x):= \tilde{\varphi}_\ell^i(r_\ell x)=\varphi_\ell^i \circ  \exp(r_\ell x)$. Then, as $\ell \rightarrow \infty,$
			the variance $\Var(\mathcal{Z}^{\varphi_{\ell}^i}(T_\ell))$ is given by
			\begin{equation}\label{var smooth}
			\Var(\mathcal{Z}^{{\varphi}_{\ell}^i}(T_\ell))=\dfrac{||\tilde{\varphi}_\ell^i||_{L^2(\tilde{B}_1)}^2}{256\pi} \cdot r_\ell^2 \log(r_\ell \ell)+O_{||\tilde{\varphi}_\ell^i||_{\infty},V(\tilde{\varphi}_\ell^i)}(r_\ell^2),
			\end{equation}
			denoting $V(\varphi)$ the total variation of a test function $\varphi$.
		\end{prop}
	The computations of the variance in Proposition \ref{var for cont func} will follow from the analysis of the integral of the two-point correlation function and $W^{\tilde{\varphi}_\ell^i}$; the main contribution will actually be given from points far from the diagonal $x=y$.\\
	To take the limit in (\ref{var smooth}) and obtain the result in Theorem \ref{varianza nodal}, we need to check that if $\varphi_{\ell}^i$ approximates $1_{B_{r_\ell}}$, as $i \rightarrow \infty$, the corresponding statement holds for the random variables $\mathcal{Z}^{\varphi_{{\ell}}^i}, \mathcal{Z}_{\ell,r_\ell}$  and their variance.	It is easy to see that, if $\varphi^i \rightarrow \varphi$ in $L^1({S}^2)$, then for every fixed $\ell$, we also have 
	\begin{equation}
	{E}[\mathcal{Z}^{\varphi^i}(T_\ell)] \rightarrow {E}[\mathcal{Z}^{\varphi}(T_\ell)];
	\end{equation} 
	indeed, it follows from the expected value of a linear statistic,
		\begin{equation}\label{media}
		{E}[\mathcal{Z}^{\varphi}(T_\ell)]=\dfrac{\int_{{{S}^2}} \varphi(x)\,dx}{2^{3/2}} \sqrt{\ell(\ell+1)}
		\end{equation}
		(\cite{W} Proposition 1.4, starting from
		(121)). We will see that the analogous result holds for the variance in view of Proposition \ref{prop5.1}.

	\begin{prop}\label{prop5.1}
		We have that, as $\ell \rightarrow \infty,$
			$${E}[\mathcal{Z}^{\varphi_{\ell}^i}(T_\ell)^2]=O( \ell r_{\ell}^4 ||\tilde{\varphi}_\ell^i||_{L^1(\tilde{B}_1)}||\tilde{\varphi}_\ell^i||_{\infty}).$$
\end{prop}
Another question is that, when applying Proposition \ref{var for cont func} for $\varphi_\ell^i$, one needs to control the error term in (\ref{var smooth}) (which may a priori depend on $\varphi_{\ell}^i$). Since we manage to control it in terms of its $L^\infty$ norm and total variation, we can solve this issue requiring $\varphi_\ell^i$ to be essentially uniformly bounded and having uniformly bounded total variation.\\\\
	The next step will be the derivation of the Central Limit Theorem, stated in Theorem \ref{CLT}. To this aim, we will start following a similar argument as in \cite{MRW}; more precisely we define first the sequence of centered random variables (``local sample trispectrum")
	\begin{equation}\label{trispectrum}
	\mathcal{M}_{\ell,r_{\ell}}:= -\dfrac{1}{4}\sqrt{\dfrac{\ell(\ell+1)}{2}}\dfrac{1}{4!} \int_{B_{r_{\ell}}} H_4(T_\ell(x)) \,dx= -\dfrac{1}{4}\sqrt{\dfrac{\ell(\ell+1)}{2}}\dfrac{1}{4!} h_{\ell,r_{\ell};4}
	\end{equation}
	where for $\ell=1,2,\dots$, 
	\begin{equation}\label{h4}
	h_{\ell,r_{\ell};4}:=\int_{B_{r_{\ell}}} H_4(T_\ell(x)) \,dx.
	\end{equation}
	The key idea is to prove the asymptotic full correlation between the ``local" nodal length and the ``local sample trispectrum":
	\begin{prop}\label{corr}
		The correlation between $\mathcal{Z}_{\ell,r_{\ell}}$ and $\mathcal{M}_{\ell,r_{\ell}}$, in the high energy limit $\ell \rightarrow \infty$, is given by
		\begin{equation}\label{corre}
		\Corr(\mathcal{Z}_{\ell,r_{\ell}};\mathcal{M}_{\ell,r_{\ell}})=1+O\bigg(\dfrac{1}{\log r_\ell \ell}\bigg)=1+o(1).
		\end{equation}
	\end{prop}
	This result requires the evaluation of the variance of $\mathcal{M}_{\ell,r_\ell}$.
	\begin{prop}\label{lemma2.3}
		The variance of $\mathcal{M}_{\ell,r_{\ell}}$ is, 	as $\ell \rightarrow \infty$, given by
		$$\Var[\mathcal{M}_{\ell,r_{\ell}}]=\dfrac{1}{256}r_{\ell}^2 \log r_{\ell} \ell+O(r_{\ell}^2).$$
	\end{prop}
	The strategy of the proof is the same as for the variance of $\mathcal{Z}_{\ell,r_{\ell}}$; hence, for $\varphi_\ell^i$ a sequence of $C^1$ functions satisfying (\ref{properties}), we define the sequence of centered random variables
	\begin{equation}\label{smoothtrisp}
	\mathcal{M}^{\varphi_{{\ell}}^i}:=-\dfrac{1}{4}\sqrt{\frac{\ell(\ell+1)}{2}}\frac{1}{4!} \int_{{S}^2} \varphi_{{\ell}}^i(y) H_4(T_\ell(y))\,dy
	\end{equation}
 and we prove the following propositions.
	
	\begin{prop}\label{lemmasmooth2}
		The variance of $\mathcal{M}^{\varphi_{\ell}^i}$, as $\ell \rightarrow \infty$, is given by
		\begin{equation}\label{M}
		\Var[\mathcal{M}^{\varphi_{\ell}^i}]=\dfrac{||\tilde{\varphi}_\ell^i||_{L^2(\tilde{B}_1)}^2}{256\pi}r_{\ell}^2 \log r_{\ell} \ell+O_{||\tilde{\varphi}_\ell^i||_{\infty},V(\tilde{\varphi}_\ell^i)}(r_{\ell}^2).
		\end{equation}
	\end{prop}
	\begin{prop}\label{propM}
		We have that, as $\ell \rightarrow \infty$,
$${E}[\mathcal{M}^{\varphi_{\ell}^i}(T_\ell)^2]=O\bigg(r_{\ell}^2\log(r_{\ell}  \ell )||\tilde{\varphi}_\ell^i||_{\infty}||\tilde{\varphi}_\ell^i||_{L^1(\tilde{B}_1)} \bigg).$$
	\end{prop}
	In view of the orthogonality of the projections, the result in (\ref{corre}) implies that the fourth chaotic component is the leading term of the chaos expansion of $\mathcal{Z}_{\ell,r_\ell}$ and hence it is sufficient to study its asymptotic behavior. In particular, exploiting the Stein-Malliavin approach (see \cite{Peccati Nourdin}), it is enough to focus on the behavior of their fourth order cumulant (\cite{Peccati Nourdin}, Theorem 5.2.7). Here, it is important to note that our argument is quite different from the proof given by \cite{MRW}; in particular, in the full sphere the behavior of the fourth-order cumulant was already established by means of Clebsch-Gordan coefficients: the latter cannot be used here due to the lack of analogous explicit results on subdomains. Hence, we derive efficient bounds by a careful exploitation of Hilb's asymptotics for powers of Legendre polynomials.\\\\
	From now on we will denote with $B_r \subset {S}^2$ the ball of radius $r$, $0<r<\pi$ centered in $N$ and with $\tilde{B}_{r}$ the disc of radius $r$ in ${R}^2$.
	
	\section{Auxiliary functions}\label{Section-auxiliary}
	In this section we introduce the auxiliary functions, announced in Section \ref{Section-ontheproof}, involved into the proofs of our main results. \\
	The indicator function $1_{B_{r_\ell}}$ belongs to the space $BV({S}^2) \cap L^{\infty}({S}^2)$; to make some computations easier, it is more convenient to deal with continuously differentiable functions.
	In order to control the error term of the variance for the approximating functions (and thus pass to the limit), it is sufficient that $\varphi_\ell^i$ is uniformly bounded and with uniformly bounded variation (see \cite{W}) and to prove that the same conditions still hold for $\tilde{\varphi}_\ell^i$, obtained through the exponential map. 
	In \cite{W} the existence of such a sequence was established. Denoting with $V(\varphi)$ the total variation of a test function $\varphi$, let consider $\{\varphi_{{\ell}}^i\}_i$ a sequence of $C^{\infty}$ functions such that, as $\ell \rightarrow +\infty$,
	\begin{equation}\label{properties}
	\begin{split}
	& \varphi_{{\ell}}^i \rightarrow 1_{B_{r_\ell}} \mbox{ in }L^1({S}^2), \\
	& V(\varphi_{\ell}^i) \rightarrow V(1_{B_{r_{\ell}}}) \mbox{ and }\\
	& ||\varphi_{\ell}^i||_{\infty} \leq ||1_{B_{r_{\ell}}}||_{\infty}.
	\end{split}
	\end{equation}
	Our goal is to check whether analogous conditions still hold for $\tilde{\varphi}_{\ell}^i= \varphi_{\ell}^i \circ \exp$, defined on ${R}^2$. To simplify the notation we set
	$\tilde{\varphi}_\ell^i(x):=\tilde{\varphi}_{\ell}^i(r_{\ell} x), x \in {R}^2$. Note that, since $\varphi_{\ell}^i$ has support on ${S}^2$, which is compact, it follows that $\tilde{\varphi}_\ell^i$ has compact support in $\tilde{B}_1$. Hence, it is easy to prove the validity of the lemma below.
	
	\begin{lem} \label{lemma}
		Let $\tilde{\varphi}_\ell^i(x):=\tilde{\varphi}_{\ell}(r_{\ell} x), x \in {R}^2$, where $\tilde{\varphi}_{\ell}^i= \varphi_{\ell}^i \circ \exp$ and  $\{\varphi_{\ell}^i\}_i $ a sequence which satisfies (\ref{properties}). Then, $\tilde{\varphi}_{{\ell}}^i: {R}^2 \rightarrow {R} $ are continuously differentiable functions such that, as $i \rightarrow \infty,$
		\begin{equation}\label{cond}
		\begin{split}
		&\tilde{\varphi}_\ell^i \rightarrow 1_{\tilde{B}_1} \mbox{ in } L^1({R}^2)
		\\&V(\tilde{\varphi}_\ell^i) \rightarrow V(1_{\tilde{B}_1})
		\\&||\tilde{\varphi}_\ell^i||_{\infty} \leq ||1_{\tilde{B}_1}||_{\infty}.
		\end{split}
		\end{equation}
	\end{lem}
	Now, let $\varphi_{\ell}: {S}^2 \rightarrow {R}$ be the indicator function $1_{B_{r_\ell}}$, $\forall \ell$. We denote $W_{r_{\ell}}(\cdot)$ the function defined in (\ref{Wdef}) with this choice of $\varphi_{{\ell}}$ and $\tilde{W_1}(\cdot)$ the one in (\ref{W_1}).
	
	\begin{lem}
	Let us consider the sequence $\varphi_{{\ell}}^i$ satisfying (\ref{properties}), $W^{\varphi_{\ell}^i}(\cdot)$ and $\tilde{W}^{\tilde\varphi_\ell^i}(\cdot)$ defined as (\ref{Wdef}) and (\ref{W_1}), respectively; then
		\begin{equation}\label{first}
		W^{\varphi_{{\ell}}^i}(\rho)= \frac{\rho}{4\pi}r_\ell^2||\tilde{\varphi}^i_\ell||^2_{L^2(\tilde{B}_1)}+O_{||\tilde{\varphi}_\ell^i||_{\infty},V(\tilde{\varphi}_\ell^i)}(\rho^2 r_\ell)+O(\rho^3 r_\ell^2 ||\tilde{\varphi}_\ell^i||_{L^2(\tilde{B}_1)}^2)
		\end{equation}
		and 
		\begin{equation}\label{second}
		W^{{\varphi}_\ell^i}(\rho)=O(\rho r_\ell^2 ||\tilde{\varphi}_\ell^i||^2_{\infty}).
		\end{equation}
	\end{lem}
\begin{proof}
	As already stated in Section \ref{Section-ontheproof}, it is quite simple, and it can be found in the supplemental article \cite{supp}, Lemma B.3, to establish the asymptotic geometric relation between $W_{r_{\ell}}$ and $\tilde{W_{1}}$, given in (\ref{1111}).
	If we consider the sequence $\varphi_{{\ell}}^i$ satisfying (\ref{properties}), $W^{\varphi_{\ell}^i}(\cdot)$ and $\tilde{W}^{\tilde\varphi_\ell^i}(\cdot)$ defined as (\ref{Wdef}) and (\ref{W_1}), respectively, it is easy to see that (\ref{1111}) holds for $W^{\varphi_{\ell}^i}$ and $\tilde{W}^{\tilde{\varphi}_\ell^i}$; namely, as $\ell \rightarrow \infty$,
	\begin{equation}\label{Rel1}
	W^{\varphi_{\ell}^i}(\rho)=r_\ell^3 \tilde{W}^{\tilde\varphi_\ell^i}(\rho \frac{1}{r_\ell})(1+O(\rho^2)),
	\end{equation}
		uniformly for $\rho \in [0,2r_\ell]$
(for the proof see the supplement article \cite{supp}, Section B.3, Corollary B.4).\\
	We can also get further informations on $\tilde{W}^{\tilde{\varphi}_\ell^i}$, e.g.,
	using polar coordinates with centre $x$, for each $x \in {R}^2$, (e.g. $y=(y_1,y_2) \rightarrow (\zeta,\phi)$ with $\zeta=\rho$ and $\phi=\arctan \frac{y_2-x_2}{y_1-x_1}$) we write
	\begin{equation*}
	\begin{split}
	\tilde{W}^{\tilde\varphi_\ell^i}(\rho)&= \frac{1}{8\pi^2} \int_{d(x,y)=\rho} \tilde{\varphi}_\ell^i (x) \tilde{\varphi}_\ell^i (y) \,dxdy
	=\dfrac{\rho}{8\pi^2} \int_{{R}^2} \tilde{\varphi}_\ell^i(x) \int_{0}^{2\pi} \tilde{\varphi}^i_{\ell,x}(\rho \cos \phi, \rho \sin \phi) \,d\phi  \,dx
	\end{split}
	\end{equation*}
	for a suitable defined function $\tilde{\varphi}^i_{\ell,x}: R^2 \rightarrow \{0,1\}.$
Defining
	\begin{equation}\label{W_0}
	\tilde{W}_{0}^{\tilde{\varphi}_\ell^i}(\rho):=  \int_{{R}^2} \tilde{\varphi}_\ell^i(x)  \int_{0}^{2\pi}  \tilde{\varphi}^i_{\ell,x}(\rho \cos \phi, \rho \sin \phi) \,d\phi \,dx,
	\end{equation}
	we have that 
	\begin{equation}\label{W1}
	\tilde{W}^{\tilde{\varphi}_\ell^i}(\rho)=\dfrac{\rho}{8\pi^2} \tilde{W_{0}}^{\tilde{\varphi}_\ell^i}(\rho).
	\end{equation}
	Note that $\tilde{W}_{0}^{\tilde{\varphi}_\ell^i}(\rho)$ is bounded by
	\begin{equation}\label{Rel3}
	|\tilde{W}_{0}^{\tilde{\varphi}_\ell^i}(\rho)|\leq 2\pi||\tilde{\varphi}_\ell^i||_{\infty}||\tilde{\varphi}_\ell^i||_{L^1(\tilde{B}_1)}\leq 2\pi^2||\tilde{\varphi}_\ell^i||_{\infty}^2,
	\end{equation}
	and in zero, it is equal to
	\begin{equation}\label{W2}
	\tilde{W_{0}}^{\tilde{\varphi}_\ell^i}(0)=2\pi ||\tilde{\varphi}_\ell^i||_{L^2(\tilde{B}_1)}^2.
	\end{equation}
	Moreover, it can be seen that the derivative of $	\tilde{W}_{0}^{\tilde{\varphi}_\ell^i}(\rho)$ is uniformly bounded by
	\begin{equation}\label{Rel4}
	|\tilde{W}^{\tilde{\varphi}_\ell^i}_0(\rho)'|\leq 2\pi||\tilde{\varphi}_\ell^i||_{\infty} V(\tilde{\varphi}_\ell^i);
	\end{equation}
	indeed, exchanging the order of the derivative and the integral, we obtain
	\begin{equation*}
	\begin{split}
	\bigg|\dfrac{\partial}{\partial \rho} \tilde{W}_{0}^{\tilde{\varphi}_\ell^i}(\rho)\bigg| & \leq  \int_{{R}^2} \bigg|\tilde{\varphi}_\ell^i(x)    \int_{0}^{2\pi} \dfrac{\partial}{\partial\rho}\tilde{\varphi}_\ell^i(\rho \cos \phi, \rho \sin \phi) \,d \phi \bigg|\,dx\\& \leq 2\pi \int_{\tilde{B}_1} |\tilde{\varphi}_\ell^i(x)| ||\nabla \tilde{\varphi}_\ell^i(x)|| \,dx
	=2\pi ||\tilde{\varphi}_\ell^i||_{\infty}V(\tilde{\varphi}_\ell^i).
	\end{split}
	\end{equation*}
Then, in view of (\ref{W2}) and the continuous differentiability of $\tilde{W}_{0}^{\tilde{\varphi}_\ell^i}(\rho)$, the Mean Value Theorem implies that,
		as $\rho \rightarrow 0$,
	\begin{equation}\label{Alessio2} 
	\tilde{W}_{0}^{\tilde{\varphi}_\ell^i}(\rho)=2\pi ||\tilde{\varphi}_\ell^i||_{L^2(\tilde{B}_1)}^2+O_{||\tilde{\varphi}_\ell^i||_{\infty},V(\tilde{\varphi}_\ell^i)}(\rho).
	\end{equation} 
Now, putting (\ref{W1}) in (\ref{Rel1}) we can state that, as $\ell \rightarrow \infty$,
		\begin{equation}\label{AP}
	W^{\varphi_{\ell}^i}(\rho)=\frac{\rho}{8\pi^2} r_\ell^2 \tilde{W}^{\tilde\varphi_\ell^i}_0\left(\frac{\rho}{r_\ell}\right)(1+O(\rho^2)),
\end{equation} 
with $\rho \in [0,r_\ell].$\\
Finally, replacing (\ref{Alessio2}) in (\ref{AP}), we obtain (\ref{first}) and then thanks also to (\ref{Rel3}), (\ref{second}) follows.
\end{proof}

	From now on $\{ \varphi_\ell^i\}_i$ will denote a sequence satisfying (\ref{properties}) and $\{\tilde{\varphi}_\ell^i\}_i$ the one satisfying (\ref{cond}).
	
	\section{Proof of the main results}\label{section-proofs}
	\subsection{Proof of Theorem \ref{varianza nodal} (Asymptotics for the variance)}\label{section-proofvar}
	As we have already mentioned, we apply an approximation argument; hence assuming the validity of Proposition \ref{var for cont func} and Proposition \ref{prop5.1} we prove Theorem \ref{varianza nodal}.
	\begin{proof}[Proof of Theorem \ref{varianza nodal} assuming Proposition \ref{var for cont func} and Proposition \ref{prop5.1} ]
		Let $\varphi_{{\ell}}^i \in C^\infty$ be a sequence of smooth functions satisfying (\ref{properties}) and let $\tilde{\varphi}_\ell^i$ defined as in Lemma \ref{lemma} satisfying (\ref{cond}). Proposition \ref{var for cont func} states that
	\begin{equation}\label{2.8bis}
	\Var(\mathcal{Z}^{\varphi_{\ell}^i}(T_\ell))=\dfrac{||\tilde{\varphi}_\ell^i||_{L^2(\tilde{B}_1)}^2}{256\pi} \cdot r_{\ell}^2 \log(r_{\ell}  \ell )+O_{||\tilde{\varphi}_\ell^i||_{\infty},V(\tilde{\varphi}_\ell^i)}(r_{\ell}^2);
	\end{equation}
		since $\tilde{\varphi}_\ell^i$ and $1_{\tilde{B}_1}$ are uniformly bounded, $L^1({R}^2)$-convergence implies $L^2({R}^2)$-convergence,
		$||\tilde{\varphi}_\ell^i||_{L^2({R}^2)}\rightarrow ||1_{\tilde{B}_1}||_{L^2({R}^2)}=\sqrt{\pi}$
		and it remains to prove that
		$$\Var[\mathcal{Z}^{\varphi_{\ell}^i}(T_\ell)] \rightarrow \Var[\mathcal{Z}_{\ell,r_{\ell}}].$$
To take the limit we need to show that the distribution of $\mathcal{Z}^{\varphi_{{\ell}}^i}$ depends continuously on $\varphi_\ell^i$.
Indeed, by linearity of $\mathcal{Z}^\varphi$ on $\varphi$, we have that
		\begin{equation*}
		\begin{split}
		{E}[(\mathcal{Z}^{\varphi_{\ell}^i}(T_\ell)-\mathcal{Z}_{\ell,r_{\ell}})^2]
		={E}[(\mathcal{Z}^{\varphi_{\ell}^i-1_{B_{r_\ell}}})^2]
			\end{split}
		\end{equation*}
		and applying Proposition \ref{prop5.1} to the difference $\varphi_{{\ell}}^i-1_{B_{r_\ell}}$, we get that
			\begin{equation*}
		\begin{split}
	&{E}[(\mathcal{Z}^{\varphi_{\ell}^i-1_{B_{r_\ell}}})^2] 
		=O({\ell}r_{\ell}^4 ||\tilde{\varphi}_\ell^i-1_{\tilde{B}_1}||_{L^1(\tilde{B}_1)} ||\tilde{\varphi}_{\ell}^i-1_{\tilde{B}_1}||_{\infty}) \rightarrow 0,
		\end{split}
		\end{equation*}
		as $i \rightarrow \infty,$ hence
		\begin{equation*}
		\begin{split}
		|\Var[\mathcal{Z}^{\varphi_{\ell}^i}(T_\ell)] - \Var[\mathcal{Z}_{\ell,r_{\ell}}]| &=
		|E[(\mathcal{Z}^{\varphi_{\ell}^i}(T_\ell))^2-(\mathcal{Z}_{\ell,r_{\ell}})^2]|+ |
		({E}[\mathcal{Z}^{\varphi_{\ell}^i}(T_\ell)])^2-(E[\mathcal{Z}_{\ell,r_{\ell}}])^2|.
		\end{split}
		\end{equation*}
	The second summand goes to zero for (\ref{media}), whereas for the first summand we have that
	\begin{equation}\label{ulla}
		\begin{split}
	& |E[(\mathcal{Z}^{\varphi_{\ell}^i}(T_\ell)^2-(\mathcal{Z}_{\ell,r_{\ell}})^2)]|
	=|{E}[(\mathcal{Z}^{\varphi_{\ell}^i}(T_\ell)-\mathcal{Z}_{\ell,r_{\ell}})^2]-2E[\mathcal{Z}_{\ell,r_\ell}^2]+2{E}[\mathcal{Z}^{\varphi_{\ell}^i}(T_\ell) \mathcal{Z}_{\ell,r_{\ell}}]|\\&
	\leq |E[(\mathcal{Z}^{\varphi_{\ell}^i}(T_\ell)-\mathcal{Z}_{\ell,r_{\ell}})^2]|
		+2|{E}[\mathcal{Z}^{\varphi_{\ell}^i}(T_\ell)\mathcal{Z}_{\ell,r_{\ell}}  -(\mathcal{Z}_{\ell,r_\ell})^2]|\\&
		\leq |E[(\mathcal{Z}^{\varphi_{\ell}^i}(T_\ell)-\mathcal{Z}_{\ell,r_{\ell}})^2]|
		+2|{E}[(\mathcal{Z}^{\varphi_{\ell}^i}(T_\ell) -(\mathcal{Z}_{\ell,r_\ell})) \mathcal{Z}_{\ell, r_\ell}]|\\&
		\leq |E[(\mathcal{Z}^{\varphi_{\ell}^i}(T_\ell)-\mathcal{Z}_{\ell,r_{\ell}})^2]|
		+2{E}[(\mathcal{Z}^{\varphi_{\ell}^i}(T_\ell) -\mathcal{Z}_{\ell,r_\ell})^2]^{1/2} E[\mathcal{Z}_{\ell,r_\ell}^2)]^{1/2}
		\end{split}
		\end{equation}
	which goes to zero for Proposition \ref{prop5.1}.
	Hence, taking the limit, as $i \rightarrow \infty$, in (\ref{2.8bis}) we obtain the thesis of Theorem \ref{varianza nodal}.
	\end{proof}
	
	Before proving Proposition \ref{var for cont func} and Proposition \ref{prop5.1} we introduce the 2-point correlation function $\tilde{K}_{\ell}(x,y)=\tilde{K}_\ell(d(x,y))$, defined as
	$$\tilde{K}_\ell(x,y)=\dfrac{1}{(2\pi ) \sqrt{1-P_\ell(x,y)^2}}{E}[||\nabla T_\ell(x)|| \cdot ||\nabla T_\ell(y)|| | T_\ell(x)=T_\ell(y)=0]$$
	 (see \cite{W}). The following result is proved in \cite{W}, Proposition 3.5. 
	 \begin{prop}\label{2-point corr}
	 	For any choice of $C>0$, as $\ell \rightarrow \infty$, we have
	 	\begin{equation}\label{56}
	 	\begin{split}
	 	K_{\ell}(\psi)&= \frac{1}{4} +\dfrac{1}{2} \dfrac{\sin(2\psi)}{\pi \ell \sin(\psi/L)} + \dfrac{1}{256} \dfrac{1}{\pi^2 \ell \sin(\psi/L)\psi}+ \dfrac{9}{32} \dfrac{\cos(2\psi)}{\pi \ell \psi \sin(\psi/L)}+\\ &+\dfrac{\frac{27}{64}\sin(2\psi)-\frac{75}{256}\cos (4\psi)}{\pi^2 \ell \psi \sin(\psi/L)}+O\bigg(\frac{1}{\psi^3}+\frac{1}{\ell \psi}\bigg),
	 	\end{split}
	 	\end{equation}
	 	uniformly for $C<\psi<\dfrac{\pi L}{2}$, where $\tilde{K}_\ell(\frac{\psi}{L})= \frac{\ell(\ell+1)}{2} K_\ell(\psi)$.
	 \end{prop}
	 It is also known that, for $0<\psi<C$, we may bound (\cite{W}, equation (98)) 
	 \begin{equation}\label{K per piccolo argomento}
	 \bigg| K_\ell(\psi)\bigg|=O\left(\frac{1}{\psi}\right).
	 \end{equation}
	\begin{proof}[Proof of Proposition \ref{var for cont func}]
		In \cite{W} (Proof of Theorem 1.4), it is shown that for functions in $C^1({S}^2)$, it is possible to write
		\begin{equation}\label{fu}
		{E}[(\mathcal{Z}^{\varphi^{i}_{\ell}}(T_\ell))^2]= \int_{{S}^2\times {S}^2} \varphi_{\ell}^{i}(x) \varphi_{\ell}^{i}(y) \tilde{K}_{\ell}(x,y) \,dx dy.
		\end{equation}
	Employing Fubini, we get
		\begin{equation}\label{mom}
		{E}[(\mathcal{Z}^{\varphi^{i}_{\ell}}(T_\ell))^2]= 8\pi^2 \int_{0}^{2r_\ell} \tilde{K}_{\ell}(\rho) W^{\varphi_{\ell}^i}(\rho) \,d\rho;
		\end{equation}
		with $\tilde{K}_{\ell}(\rho)=\tilde{K}_{\ell}(x,y),$
		$x,y \in {S}^2$ being any pair of points with $d(x,y)=\rho.$
Indeed, we change coordinates in (\ref{fu}), centering $x$ and parameterizing $y$ in terms of $(l, \theta)$, where $l=d(x,y) \in [0,2r_\ell]$ is the distance between $x$ and $y$ and $\theta \in [0,2\pi]$. The norm of the Jacobian of this change of coordintaes is $1$, since every transformation in the sphere can be seen as a rotation; then, applying Fubini and doing the same change of coordinates in the definition of $W^{\varphi_{{\ell}}^i}$ (this time $l$ is fixed inside the integral to be $\rho$) it is seen the validity of equation (\ref{mom}).
	Now, denoting $L=\ell+\frac{1}{2}$, changing the coordinates $\rho= \frac{\psi}{L}$, and writing $8\pi^2$ as $2\pi |S^2|$, we have that
		$${E}[(\mathcal{Z}^{\varphi^{i}_{\ell}}(T_\ell))^2]= \frac{2\pi |{S}^2|}{L} \int_{0}^{2r_\ell L} \tilde{K}_{\ell}\left(\frac{\psi}{L}\right) W^{\varphi_{\ell}^i}\left(\frac{\psi}{L}\right) \,d\psi;$$
		setting $\tilde{K_{\ell}}(\frac{\psi}{L}):=\frac{\ell(\ell+1)}{2}K_{\ell}(\psi)$,
		we obtain that
		$${E}[(\mathcal{Z}^{\varphi^{i}_{\ell}}(T_\ell))^2]= \frac{\pi |{S}^2|}{L} \ell(\ell+1) \int_{0}^{2r_\ell L} K_{\ell}(\psi) W^{\varphi_{\ell}^i}\left(\frac{\psi}{L}\right) \,d\psi.$$ 
		Moreover, from (\ref{media}) it follows that $$(E[\mathcal{Z}^{\varphi_\ell^i}(T_\ell)])^2= \frac{\ell(\ell+1)}{2 \cdot 4} \int_{{{S}^2}\times {{S}^2}} \varphi_\ell^i(x) \varphi_\ell^i(y)\, dx dy$$ and applying Fubini and changing cordinates as above we obtain $$E[\mathcal{Z}^{\varphi_\ell^i}(T_\ell)]^2=\frac{\ell(\ell+1)}{8L} 2\pi |S^2| \int_{0}^{2 r_\ell L} W^{\varphi_{{\ell}}^i}\left(\frac{\psi}{L}\right) \,d\psi,$$
		from which we conclude that
		\begin{equation}\label{ics}
		\Var[\mathcal{Z}^{\varphi^{i}_{\ell}}(T_\ell)]=\dfrac{\pi|{S}^2|\ell(\ell+1)}{L} \int_{0}^{2r_\ell L} \left(K_{\ell}(\psi)-\frac{1}{4}\right) W^{\varphi_{\ell}^i}\left(\frac{\psi}{L}\right) \, d\psi.
		\end{equation}		
	Splitting the interval of the integral in $[0,1]$ and $[1, 2r_\ell L]$, we have that
		\begin{equation}\label{ale}
	\begin{split}
	\Var[(\mathcal{Z}^{\varphi^{i}_{\ell}}(T_\ell))]&=\dfrac{\pi|{S}^2|\ell(\ell+1)}{L} \int_{0}^{1} \left(K_{\ell}(\psi)-\frac{1}{4}\right) W^{\varphi_{\ell}^i}\left(\frac{\psi}{L}\right) \, d\psi\\&
	+\dfrac{\pi|{S}^2|\ell(\ell+1)}{L} \int_{1}^{2r_\ell L} \left(K_{\ell}(\psi)-\frac{1}{4}\right) W^{\varphi_{\ell}^i}\left(\frac{\psi}{L}\right) \, d\psi
	\end{split}
	\end{equation}
and in view of (\ref{K per piccolo argomento}) and (\ref{second}) the first integral in (\ref{ale}) is equal to 
		\begin{equation*}
		\begin{split}
		&O_{||\tilde{\varphi}_\ell^i||_{\infty}}\bigg( \dfrac{\pi|{S}^2|\ell(\ell+1)}{L^2} r_\ell^2 \int_{0}^{1} \bigg|\frac{1}{\psi}-\frac{1}{4}\bigg|\psi \, d\psi\bigg)
		=O_{||\tilde{\varphi}_\ell^i||_{\infty}}( r_\ell^2).
		\end{split}
		\end{equation*}
The second integral in the right hand side of (\ref{ale}) is, exploiting $(\ref{first})$, given by
\begin{equation}
\begin{split}
&	\underbrace{\dfrac{ \pi |S^2| ||\tilde{\varphi}_\ell^i||_{L^2(\tilde{B}_1)}^2 \ell(\ell+1)}{ 4 \pi L^2} r_{\ell}^2 \int_{1}^{2r_{\ell}L} \left(K_{\ell}(\psi)-\frac{1}{4}\right) \psi \,d\psi}_{(a)}+\\&+\underbrace{O_{||\tilde{\varphi}_\ell^i||_{\infty},V(\tilde{\varphi}_\ell^i)} \bigg( \dfrac{\ell(\ell+1)}{L^3} r_{\ell} \int_{1}^{2r_{\ell}L} \left(K_{\ell}(\psi)-\frac{1}{4}\right) \psi^2 \,d\psi \bigg)}_{(b)}+\\&
+ \underbrace{O_{||\tilde{\varphi}_\ell^i||_{\infty}} \bigg( \dfrac{\ell(\ell+1)}{L^4} r_{\ell}^2 \int_{1}^{2r_{\ell}L} \left(K_{\ell}(\psi)-\frac{1}{4}\right) \psi^3 \,d\psi \bigg)}_{(c)}.
	\end{split}
\end{equation}

	Thanks to Corollary \ref{corollary}, equation (\ref{lemmaKpsi}), integral (a) is given by
		$$(a)=\dfrac{||\tilde{\varphi}_\ell^i||_{L^2(\tilde{B}_1)}^2}{256\pi} r_{\ell}^2 \log r_\ell \ell+O(r_{\ell}^2).$$
In view of Lemma \ref{lemmaK}, the error term in $(b)$ is
		\begin{equation}
		\begin{split}
		(b)&=O_{||\tilde{\varphi}_\ell^i||_{\infty},V(\tilde{\varphi}_\ell^i)}\bigg( \frac{r_\ell}{\ell} (2r_\ell \ell-1) \bigg) =O_{||\tilde{\varphi}_\ell^i||_{\infty},V(\tilde{\varphi}_\ell^i)}(r_{\ell}^2).
		\end{split}
		\end{equation}
Regarding $(c)$, similar computations lead to $(c)=O(\frac{r_\ell}{\ell^2} ((2r_\ell \ell)^2 -1))=O_{||\tilde{\varphi}_\ell^i||_{\infty}}(r_{\ell}^4)$ and then, we can conclude that the variance 
		of $\mathcal{Z}^{\varphi_{\ell}^{i}}(T_\ell)$ is
		$$\Var[\mathcal{Z}^{\varphi_{\ell}^i}(T_\ell)]=\frac{||\tilde{\varphi}_\ell^i||_{L^2(\tilde{B}_1)}^2}{256\pi} r_{\ell}^2 \log(r_{\ell}  \ell )+O_{||\tilde{\varphi}_\ell^{i}||_{\infty},V(\tilde{\varphi}_\ell^{i})}(r_{\ell}^2).$$
	\end{proof}
	
	\begin{proof}[Proof of Proposition \ref{prop5.1}]
		As we did in the proof of Proposition \ref{var for cont func} we write
		\begin{equation}
		\begin{split}
		{E}[(\mathcal{Z}^{\varphi_{\ell}^i}(T_\ell))^2]&= \!\int_{{S}^2 \times {S}^2}\! \!\!\tilde{K_{\ell}}(x,y) \varphi_{\ell}^i(x) \varphi_{\ell}^i(y)\,dxdy\\&
		= \dfrac{\pi|{S}^2|\ell(\ell+1)}{L} \!\int_{0}^{2r_{\ell} L} \!\!K_{\ell}(\psi) W^{\varphi_{\ell}^i}\left(\frac{\psi}{L}\right) \,d\psi
	\end{split}
	\end{equation}
Splitting the integral,	in $[1, 2r_\ell \ell]$, thanks to Lemma \ref{lemmaK}, $K_{\ell}(\psi)$ is bounded by a constant so that $|K_{\ell}(\psi)|=O_C(1)$; whereas in $[0,1]$, 
		we exploit (\ref{K per piccolo argomento}). Hence, using (\ref{first}), we get
		\begin{equation}\label{u}
		\begin{split}
		&	\bigg| \frac{r_\ell^2}{\ell} \int_{1}^{2r_{\ell} L} \psi K_{\ell}(\psi) ||\tilde{\varphi}_\ell^i||_{L^2(\tilde{B}_1)}^2  \,d\psi \bigg| 
		\ll \frac{L^2}{\ell} r_{\ell}^4 ||\tilde{\varphi}_\ell^i||_{L^2(\tilde{B}_1)}^2 \int_{0}^{2} \rho\,d\rho \ll L r_{\ell}^4  ||\tilde{\varphi}_\ell^i||_{\infty}	||\tilde{\varphi}_\ell^i||_{L^1(\tilde{B}_1)}
		\end{split}
		\end{equation}
		and
		\begin{equation}
		\begin{split}
		&	\bigg|\frac{r_\ell^2}{\ell} \int_{0}^{1} \psi K_{\ell}(\psi) ||\tilde{\varphi}_\ell^i||_{L^2(\tilde{B}_1)}^2  \,d\psi\bigg| \ll \frac{r_\ell^2}{\ell}  \int_{0}^{1}||\tilde{\varphi}_\ell^i||_{L^2(\tilde{B}_1)}^2 \,d\psi
		\end{split}
		\end{equation}
which is dominated by (\ref{u}) and the thesis follows.
	\end{proof}
	
	\subsection{Proof of Theorem \ref{CLT} (Central Limit Theorem)} \label{section-proofCLT}
	We split this section in more subsections to make our argument clearer. Firstly, in \ref{correlationsubsection} we show that the nodal length and the integral of $H_4(T_\ell(x))$ in the shrinking spherical cap are fully correlated; secondly, in \ref{fourthcumulantsubsection} we compute the fourth cumulant of the ``local" sample trispectrum in order to apply the Fourth Moment Theorem (\cite{Peccati Nourdin}, Theorem 5.2.6) and to conclude the proof of the Central Limit Theorem in \ref{cltproof}. 
	\subsubsection{Correlation between $\mathcal{Z}_{\ell,r_{\ell}}$ and $\mathcal{M}_{\ell,r_{\ell}}$ (Proof of Proposition \ref{corr})}\label{correlationsubsection}
	
	Here we show the asymptotic equivalence (in the $L^2(\Omega)$-sense) of the nodal length $\mathcal{Z}_{\ell,r_{\ell}}$ and the trispectrum $\int_{B_{r_{\ell}}}H_4(T_\ell(x)) \, dx$. In \cite{MRW}, the case of the full sphere was considered and it was established that, as $\ell \rightarrow +\infty$,
	$$\Corr(\mathcal{Z}(T_\ell); \mathcal{M}_{\ell})=1+O\bigg(\frac{1}{\log \ell}\bigg),$$ where $\mathcal{M}_\ell$ is the integral of $H_4(T_\ell(x))$ on ${S}^2$. In decreasing domains the full correlation still holds. Let us define the sequence of centered random variables $\mathcal{M}_{\ell,r_{\ell}}$ as in (\ref{trispectrum}). To prove Proposition \ref{corr} we shall need Proposition \ref{lemma2.3} and the lemma below.
	\begin{lem}\label{lemma2.2}
		The covariance between $\mathcal{Z}_{\ell,r_{\ell}}$ and $\mathcal{M}_{\ell,r_{\ell}}$, as $\ell \rightarrow \infty,$  is given by
		\begin{equation}\label{cov}
		\Cov(\mathcal{Z}_{\ell,r_{\ell}};\mathcal{M}_{\ell,r_{\ell}})= \dfrac{1}{256} r_{\ell}^2 \log r_{\ell}\ell+O(r_{\ell}^2).
		\end{equation}
	\end{lem}
	Putting together Lemma \ref{lemma2.2}, Proposition \ref{lemma2.3} and Theorem \ref{varianza nodal}, Proposition \ref{corr} is easily proved:
	$$\Corr(\mathcal{Z}_{\ell,r_{\ell}};\mathcal{M}_{{\ell},r_{\ell}})=\dfrac{\Cov(\mathcal{Z}_{{\ell};r_{\ell}};\mathcal{M}_{{\ell};r_{\ell}})}{\sqrt{\Var(\mathcal{Z}_{{\ell};r_{\ell}})\Var(\mathcal{M}_{{\ell};r_{\ell}})}}=1+O\bigg(\dfrac{1}{\log r_\ell \ell}\bigg).$$

Hence, we need to prove Lemma \ref{lemma2.2} and Proposition \ref{lemma2.3}. In order to do that we define the 2-point cross correlation function $\mathcal{J}_{\ell}(\psi,4)$.  
We shall write $\bar{x}=(0,0)$ for the North Pole and $y(\theta)=(0,\theta)$ for the points on the meridian where $\varphi = 0$. Then, 
\begin{equation}\label{Jdef}
\mathcal{J}_{\ell}(\psi;4)=\left[ -\frac{1}{4} \sqrt{\dfrac{\ell(\ell+1)}{2}} \dfrac{1}{4!} \right] \times \dfrac{8\pi^2}{L} {E}\left[\Psi_{\ell}(\bar{x},4)H_4\left(T_\ell\left(y\left(\frac{\psi}{L}\right)\right)\right)\right]
\end{equation} 
(see the supplement article \cite{supp}, Section A.1, for the definition of $\Psi_{\ell}(\bar{x},4)$ and see also \cite{MRW}).
 The following result is proved in \cite{MRW}, Proposition 3.1.
\begin{prop}\label{2-point cross corr}
	For any constant $C>0$, uniformly over $\ell$ we have, for $0<\psi<C,$
	\begin{equation}\label{3.2}
	\mathcal{J}_{\ell}=O( \ell ),
	\end{equation}
	and, for $C<\psi<L \frac{\pi}{2}$,
	\begin{equation}\label{3.3}
	\begin{split}
	\mathcal{J}_{\ell}(\psi,4)= &\dfrac{1}{64} \dfrac{1}{\psi \sin(\psi/L)}+\frac{5}{64} \frac{\cos 4\psi}{\psi \sin(\psi/L)}-\frac{3}{16} \frac{\sin(2\psi)}{\psi \sin(\psi/L)}+ O\left( \dfrac{1}{\psi^2\sin(\psi/L)}\right)\\&+O \left( \frac{1}{\ell \psi\sin(\psi/L)}\right).
	\end{split}
	\end{equation}
\end{prop}

	\begin{proof}[Proof of Lemma \ref{lemma2.2}]
In the supplement article \cite{supp}, Lemma B.2, we show that 
	 $$\Cov(\mathcal{Z}_{\ell,r_\ell}, \mathcal{M}_{\ell,r_\ell})= \lim_{i \rightarrow \infty} \int_{0}^{2r_\ell \ell} \mathcal{J}_\ell (\psi,4) W^{\varphi_{{\ell}}^i}\left(\frac{\psi}{L}\right) \,d\psi,$$
	 where $\mathcal{\mathcal{J}}_{\ell}(\psi;4)$ is the two point cross-correlation function defined in (\ref{Jdef}).
		Then, to compute this integral we split it in:
		$$I_1:=\int_{0}^{1} \mathcal{J}_{\ell}(\psi,4)W^{\varphi_{\ell}^i}\bigg(\frac{\psi}{L}\bigg)\,d\psi \mbox{ }\mbox{ }\mbox{ }
		\mbox{ and } 	\mbox{ }\mbox{ }\mbox{ } I_2:=\int_{1}^{2r_{\ell}L}\mathcal{J}_{\ell}(\psi,4)W^{\varphi_{\ell}^i}\bigg(\frac{\psi}{L}\bigg)\,d\psi;$$
exploiting (\ref{second}) it follows that
		\begin{equation}
		\begin{split}
		I_1&=\int_{0}^{1} \mathcal{J}_{\ell}(\psi,4)W^{\varphi_{\ell}^i}\bigg(\frac{\psi}{L}\bigg)\,d\psi 
		\ll \frac{r_\ell^2}{\ell} \int_{0}^{1} |\mathcal{J}_\ell(\psi,4)| \psi ||{\tilde{\varphi}_\ell^i}||_{\infty}^2 \,d\psi 
		\end{split}
		\end{equation}
	and thanks to 
	(\ref{3.2}), we have, 
	$$I_1 \ll \frac{r_\ell^2}{L} \int_{0}^{1} \ell \psi ||{\tilde{\varphi}_\ell^i}||_\infty^2 \,d\psi =O_{||{\tilde{\varphi}_\ell^i}||_\infty}(r_\ell^2),$$
	as $\ell \rightarrow \infty.$
		Regarding $I_2$, equation (\ref{first}) implies
		\begin{equation}\label{ua}
	\begin{split}
	I_2=&	||\tilde{\varphi}_\ell^i||^2_{L^2(\tilde{B}_1)}	 \frac{r_\ell^2}{4\pi L}\int_{1}^{2r_{\ell}L} \mathcal{J}_\ell (\psi,4) \psi
	\,d\psi+O_{||\tilde{\varphi}_\ell^i||_{\infty},V(\tilde{\varphi}_\ell^i)}\bigg( \int_{1}^{2r_{\ell}L}  \mathcal{J}_\ell (\psi,4) r_{\ell}\frac{\psi^2}{L^2} \,d\psi \bigg)\\&
	+ O_{||{\tilde{\varphi}_\ell^i}||_\infty}
	\bigg(\int_{1}^{2r_{\ell}L}  \mathcal{J}_\ell (\psi,4) {r_{\ell}^2} \frac{\psi^3}{L^3}
\,d\psi\bigg);
	\end{split}
		\end{equation}
		thanks to Lemma \ref{lemmaJ}, it is easy to see that the second and the third terms of (\ref{ua}) are, respectively, given by
			$$O_{||\tilde{\varphi}_\ell^i||_{\infty},V(\tilde{\varphi}_\ell^i)}\bigg(\frac{r_\ell}{L}  \int_{1}^{2r_\ell \ell}P_2(\psi) \, d\psi\bigg)=O_{||\tilde{\varphi}_\ell^i||_{\infty},V(\tilde{\varphi}_\ell^i)}( r_\ell^2),$$
		and
		$$O_{||{\tilde{\varphi}_\ell^i}||_\infty} \bigg(\frac{r_\ell^2}{\ell^2} \int_{1}^{2r_\ell \ell} \frac{P_2(\psi)}{\psi^2} \psi^3\,d\psi\bigg)=O_{||{\tilde{\varphi}_\ell^i}||_\infty}\bigg(\frac{r_\ell^2}{\ell}(2r_\ell \ell-1)^2\bigg)=O_{||{\tilde{\varphi}_\ell^i}||_\infty}(r_\ell^4),$$
where $P_2(\psi)$ is defined in (\ref{PJ}).
Finally,  (\ref{lemma Jpsi}) applied to the first term of (\ref{ua}) leads to
	$$I_2= \frac{||\tilde{\varphi}_\ell^i||_{L^2(\tilde{B}_1)}^2}{256\pi} r_{\ell}^2 \log(2r_{\ell}L)+O(r_{\ell}^2)$$ and hence the conclusion of the lemma follows.
	\end{proof}
	Proposition \ref{lemma2.3} is easily seen as a corollary of Proposition \ref{lemmasmooth2} and Proposition \ref{propM} as follows.
	\begin{proof}[Proof of Proposition \ref{lemma2.3} assuming Proposition \ref{lemmasmooth2} and Proposition \ref{propM} ]
		Let $\mathcal{M}^{\varphi_{{\ell}}^i}$ defined as in (\ref{smoothtrisp}); since $\mathcal{M}^{\varphi}$ is linear in $\varphi$, we have that
	 $${E}[(\mathcal{M}^{{\varphi}^i_{\ell}}-\mathcal{M}_{{\ell},r_{\ell}})^2]
	 =	{E}[(\mathcal{M}^{{\varphi}^i_{\ell}-1_{B_{r_{\ell}}}})^2]$$ and applying Proposition \ref{propM} to the function $\varphi^i_{\ell}-1_{B_{r_{\ell}}}$ and doing similar computations we did in (\ref{ulla}),
		we get
		$$|\Var(\mathcal{M}^{{\varphi}^i_{\ell}})-\Var(\mathcal{M}_{{\ell},r_{\ell}})| \leq |{E}[(\mathcal{M}^{{\varphi}^i_{\ell}-1_{B_{r_{\ell}}}})^2]|+2|{E}[(\mathcal{M}^{{\varphi}^i_{\ell}}-\mathcal{M}_{{\ell},r_{\ell}})^2]^{1/2} E[\mathcal{M}_{{\ell};r_{\ell}}^2]^{1/2}|+$$
		$$|({E}[\mathcal{M}^{{\varphi}^i_{\ell}}])^2-({E}[\mathcal{M}_{{\ell},r_{\ell}}])^2|$$
	which goes to zero, as $i\rightarrow \infty$, by the $L^1$ convergence of $\varphi_{{\ell}}^i$ and Proposition \ref{propM}. Then, taking the limit in Proposition \ref{lemmasmooth2}, the thesis follows.
	\end{proof}
	Let us now prove Proposition \ref{lemmasmooth2} and Proposition \ref{propM}. We recall that $P_\ell$ is the covariance function of $T_\ell$ and the following expansion for $P_{\ell}(\cos \frac{\psi}{L})^4$ is given in \cite{W}, Lemma 3.9: for $\ell \geq 1$ and any constant $C>0$, $C<\psi<\pi L /2$,
		\begin{equation}\label{pl4}
		P_{\ell}(\cos (\psi/L))^4=\dfrac{\frac{3}{2}-2\sin(2\psi)-\frac{1}{2} \cos(4\psi)}{\pi^2 {\ell}^2 \sin(\psi/L)^2}+O\bigg(\dfrac{1}{\psi^3}\bigg).
		\end{equation}
		Recall also that, for $0<\psi<\frac{\pi L}{2}$, as $\ell \rightarrow \infty$,
		\begin{equation}\label{Pazero}
		\bigg|P_\ell(\cos \frac{\psi}{L})\bigg|=O\bigg(\frac{1}{\sqrt{\psi}}\bigg),
		\end{equation}
	(see (\ref{hilb2}), see also \cite{W}).
	\begin{proof}[Proof of Proposition \ref{lemmasmooth2}]
		The idea of the proof is quite similar to the one in Proposition \ref{var for cont func}; actually,
		we write the variance of $\mathcal{M}^{{\varphi}^i_{\ell}}$ as
		\begin{equation}\label{uuu}
		\begin{split}
		\Var(\mathcal{M}^{{\varphi}^i_{\ell}})&=\Var\bigg[-\dfrac{1}{4}\sqrt{\dfrac{\ell(\ell+1)}{2}}\dfrac{1}{4!} \int_{{S}^2} \varphi_{{\ell}}^i(x) H_4(T_\ell(x)) \,dx\bigg]\\&
		=\dfrac{1}{16 }\dfrac{\ell(\ell+1)}{2}\dfrac{1}{4!^2}{E}\bigg[ \int_{{S}^2} \varphi_{{\ell}}^i(x) H_4(T_\ell(x)) \,dx \int_{{S}^2} \varphi_{{\ell}}^i(y) H_4(T_\ell(y)) \,dy\bigg]=\\&
		=\dfrac{1}{16 }\dfrac{\ell(\ell+1)}{2}\dfrac{1}{4!^2} \int_{{{S}^2}\times {{S}^2}} {E}[ H_4(T_\ell(x)) H_4(T_\ell(y))] \varphi_\ell^i(x) \varphi_{{\ell}}^i(y) \,dx \,dy\\&
		=\dfrac{1}{16 }\dfrac{\ell(\ell+1)}{2}\dfrac{1}{4!^2} 4!\int_{{S}^2\times {S}^2} P_{\ell} ( \langle x,y \rangle)^4 \varphi_\ell^i(x) \varphi_{{\ell}}^i(y) \,dx dy,
		\end{split}
		\end{equation}
where in the last passage we exploited property (A.1) of the supplement article \cite{supp}.
		Employing Fubini, (\ref{uuu}) is equal to
		\begin{equation}\label{?}
\dfrac{1}{16 }\dfrac{\ell(\ell+1)}{2}\dfrac{1}{4!}  8\pi^2 \int_{0}^{2r_{\ell}} P_{\ell}(\cos \rho)^4 W^{\varphi_{\ell}^i}(\rho) \, d\rho.
		\end{equation}
	Changing variable $\rho=\dfrac{\psi}{L}$ and splitting the integral, (\ref{?}) is equal to
		\begin{equation}\label{integral}
		\begin{split}
			\dfrac{8\pi^2}{16 }\dfrac{\ell(\ell+1)}{2L}\dfrac{1}{4!} &	\int_{0}^{1} P_{\ell}\left(\cos \frac{\psi}{L}\right)^4 W^{\varphi_{\ell}^i}\left(\frac{\psi}{L}\right) \, d\psi \\&+\dfrac{8\pi^2}{16 }\dfrac{\ell(\ell+1)}{2L}\dfrac{1}{4!}
	\int_{1}^{2r_{\ell} L} P_{\ell}\left(\cos \frac{\psi}{L}\right)^4 W^{\varphi_{\ell}^i}\left(\frac{\psi}{L}\right) \, d\psi
		\end{split}
		\end{equation}
		In view of (\ref{second}), the first integral in (\ref{integral}) is 
		\begin{equation}\label{st}
		O_{||{\tilde{\varphi}_\ell^i}||_\infty} \bigg( \int_{0}^{1} P_{\ell}\left(\cos \frac{\psi}{L}\right)^4 \psi r_\ell^2 \,d\psi \bigg);
		\end{equation} we bound $|P_\ell(x)|$ with 1 and then we obtain that (\ref{st}) is $O_{||{\tilde{\varphi}_\ell^i}||_\infty} (r_\ell^2)$.
		To compute the second integral in (\ref{integral}), we exploit (\ref{first}) to get

			\begin{equation}\label{4.38}
		\begin{split}
	&	\underbrace{\dfrac{1}{16 }\dfrac{\ell(\ell+1)}{2}\dfrac{1}{4!}  8\pi^2 \frac{r_{\ell}^2}{L 4\pi} \int_{1}^{2r_{\ell} L} P_{\ell}\left(\cos \frac{\psi}{L}\right)^4 \frac{\psi}{L} ||{\tilde{\varphi}_\ell^i}||_{L^2(\tilde{B}_1)}^2  \, d\psi}_{(i)}
\\&+
\underbrace{ O_{||\tilde{\varphi}_\ell^i||_{\infty},V(\tilde{\varphi}_\ell^i)} \bigg(	\dfrac{1}{16 }\dfrac{\ell(\ell+1)}{2L}\dfrac{1}{4!} 8\pi^2 \frac{r_{\ell}}{L^2 } \int_{1}^{2r_{\ell} L} P_{\ell}\left(\cos \frac{\psi}{L}\right)^4 \psi^2 \, d\psi}_{(ii)} \\&+
	\underbrace {	O_{||{\tilde{\varphi}_\ell^i}||_\infty} \bigg( \dfrac{1}{16 }\dfrac{\ell(\ell+1)}{2}\dfrac{1}{4!} 8\pi^2 \frac{r_{\ell}^2}{L} \int_{1}^{2r_{\ell} L} P_{\ell}\left(\cos \frac{\psi}{L}\right)^4 \frac{\psi^3}{L^3} \, d\psi \bigg)}_{(iii)}.
		\end{split}
		\end{equation}
Now, the leading term is
		\begin{equation}
			\begin{split}
		(i)&
		=\dfrac{1}{16 }\dfrac{\ell(\ell+1)}{2}\dfrac{1}{4!} 2\pi ||\tilde{\varphi}_\ell^i||_{L^2(B_1)}^2  \frac{r_{\ell}^2}{L^2} \int_{1}^{2r_{\ell}L} P_\ell\left( \cos \frac{\psi}{L}\right)^4 \psi \,d\psi
			\end{split}
		\end{equation}
		and thanks to Lemma \ref{lemmaP4} and Lemma \ref{lemmaP}, $(i)$ is
		\begin{equation}
			\begin{split}
			=\frac{||\tilde{\varphi}_\ell^i||_{L^2(\tilde{B}_1)}}{256\pi} r_{\ell}^2 \log(r_{\ell}  \ell ).
			\end{split}
		\end{equation}
	With similar calculations, it is easy to verify that $(ii)$ is 	$O_{||\tilde{\varphi}_\ell^i||_{\infty},V(\tilde{\varphi}_\ell^i)}(r_\ell^2)$ and $(iii)$ is $O_{||\tilde{\varphi}_\ell^i||_{\infty}}(r_\ell^2)$
	and hence the conclusion of the proposition follows.
	\end{proof}
	We prove now Proposition \ref{propM}. 
	\begin{proof}[Proof of Proposition \ref{propM}]
		In a similar way to the proof of Proposition \ref{prop5.1}, we can write	
		\begin{equation}\label{4.44}
		\begin{split}
		{E}[\mathcal{M}^{\varphi_{\ell}^i}(T_\ell)^2]&=O \bigg( \dfrac{\ell(\ell+1)}{L} \int_{0}^{2r_{\ell} L} P_{\ell}\left(\cos \frac{\psi}{L}\right)^4 W^{\varphi_{\ell}^i} \bigg(\frac{\psi}{L}\bigg) \,d\psi \bigg).
		\end{split}
		\end{equation}
	Splitting $[0,2 r_\ell \ell]=[0,1] \cup [1,r_\ell \ell]$, for $\psi \in [0,1]$, we can bound $|P_\ell(x)|\leq 1, \forall x \in [-1,1]$ and exploiting (\ref{second}), we get that 
\begin{equation}\label{E}
O \bigg(\ell \int_{0}^{1} P_\ell^4\left(\cos \frac{\psi}{L}\right) W^{\varphi_{\ell}^i} \bigg(\frac{\psi}{L} \bigg)\, d\psi\bigg)=O\bigg( r_{\ell}^2||\tilde{\varphi}_\ell^i||_{\infty}^2 \int_{0}^{1}  \psi \,d\psi \bigg)= O\bigg({r_{\ell}^2} ||\tilde{\varphi}_\ell^i||_{\infty}\bigg). 
\end{equation}
	Moreover, (\ref{Pazero}), Lemma \ref{lemmaP} and (\ref{first}) imply that
	\begin{equation}\label{EE}
	\begin{split}
	&O \bigg(\ell \int_{1}^{2r_\ell L} P_\ell^4\left(\cos \frac{\psi}{L}\right) W^{\varphi_{\ell}^i} \bigg(\frac{\psi}{L} \bigg)\, d\psi\bigg)=O\bigg(\frac{r_{\ell}^2 \ell}{\ell} \int_{1}^{2r_{\ell} L} \psi P_{\ell}\left(\cos \frac{\psi}{L}\right)^4 ||\tilde{\varphi}_\ell^i||_{L^2(\tilde{B}_1)}^2\,d\psi \bigg)\\&=O\bigg(r_{\ell}^2 ||\tilde{\varphi}_\ell^i||_{L^2(\tilde{B}_1)}^2\int_{1}^{2r_{\ell} L } \frac{1}{\psi^2} \psi \,d\psi \bigg)=
	O\left(r_{\ell}^2 ||\tilde{\varphi}_\ell^i||_{\infty} ||\tilde{\varphi}_\ell^i||_{L^1(\tilde{B}_1)} \log r_{\ell} \ell\right).
	\end{split}
	\end{equation}
Since (\ref{E}) is dominated by (\ref{EE}), the conclusion of the Proposition follows.
	\end{proof}

	\subsubsection{Fourth cumulant of the fourth chaotic component}\label{fourthcumulantsubsection}
	
	In light of the orthogonality of the chaotic components, the full correlation between $\mathcal{Z}_{\ell,r_\ell}$ and $\mathcal{M}_{\ell,r_\ell}$ implies that $$\Corr(\mathcal{M}_{{\ell},r_{\ell}}; Proj(\mathcal{Z}_{\ell,r_{\ell}}|C_4))=1+O\bigg(\dfrac{1}{\log r_\ell \ell}\bigg).$$
		Now to establish the validity of the CLT for the sequence $\mathcal{Z}_{\ell, r_\ell}$, we prove first that it holds for $\mathcal{M}_{\ell, r_\ell}$. In order to do that we appeal to the Fourth Moment Theorem (\cite{Peccati Nourdin}, Theorem 5.2.6), which states that, for random variables belonging to a Wiener chaos it is sufficient to show that the fourth cumulant divided by the square of the variance tends to zero to conclude that the CLT holds. Hence we investigate in the lemma below the fourth cumulant of $h_{\ell,r_\ell,4}$ (defined in \ref{h4}).
	\begin{lem}\label{4cum}
		Let $h_{\ell,r_\ell,4}$ defined as (\ref{h4}), as $\ell \rightarrow \infty,$
		\begin{equation}
		\begin{split}
		cum_4\{h_{\ell,r_\ell,4}\}
		=O\bigg( \dfrac{r_\ell^4}{\ell^4} \log r_\ell \ell \bigg).
		\end{split}
		\end{equation}
	\end{lem}
	\begin{proof}
	Following \cite{M e W 2012}, in order to find a bound for the fourth cumulant of $h_{\ell,r_\ell,4}$, we need to control the following two quantities $A_1$ and $A_2$ (see the supplement article \cite{supp}, Section A.2 and \cite{M e W 2012} for details):
		$$A_1=\int_{(B_{r_{\ell}})^4} \!P_{\ell}(\langle x_1,x_2\rangle)P_{\ell}(\langle x_1,x_3\rangle)^3  P_{\ell}(\langle x_3,x_4\rangle)P_{\ell}(\langle x_2,x_4 \rangle)^3\, \mu(dx_1)\mu(dx_2)\mu(dx_3)\mu(dx_4),$$
	$$A_2=\int_{(B_{r_{\ell}})^4}\! \!\!P_{\ell}(\langle x_1,x_2\rangle)^2 P_{\ell}(\langle x_1,x_3\rangle)^2  P_{\ell}(\langle x_3,x_4\rangle)^2P_{\ell}(\langle x_2,x_4 \rangle)^2\, \mu(dx_1)\mu(dx_2)\mu(dx_3)\mu(dx_4),$$
	where $\mu (dx_{i})$ denotes Lebesgue measure on the sphere.
	Let us focus on $A_1$; its absolute value is bounded by 
\begin{equation}\label{n}
\begin{split}
	\int_{(B_{r_{\ell }})^{4}}\left\vert P_{\ell }(\cos d(x_{1},x_{2}))\right\vert &
\left\vert P_{\ell }(\cos d(x_{1},x_{3}))^{3}\right\vert \left\vert P_{\ell
	}(\cos d(x_{3},x_{4}))\right\vert \times \\& \times \left\vert P_{\ell }(\cos
	d(x_{2},x_{4}))^{3}\right\vert \mu (dx_{1})\mu (dx_{2})\mu (dx_{3})\mu (dx_{4}).
	\end{split}
\end{equation}
 Arguing as in \cite{NPR}, we use the inequality: $x^a y^b \leq x^{a+b}+y^{a+b}$, where $x,y$ are positive, to obtain that $(\ref{n})$ can be bounded by 
	\begin{equation}\label{n1}
	\begin{split}
&	\int_{(B_{r_{\ell }})^{4}}\left\vert P_{\ell }(\cos d(x_{2},x_{4}))\right\vert
^{3}\left\vert P_{\ell }(\cos d(x_{3},x_{4}))\right\vert \left\vert P_{\ell
}(\cos d(x_{1},x_{2}))\right\vert ^{4}\mu (dx_{1})\mu (dx_{2})\mu (dx_{3})\mu
(dx_{4})\\&
+ 	\int_{(B_{r_{\ell }})^{4}}\! \!\!\!\left\vert P_{\ell }(\cos d(x_{2},x_{4}))\right\vert
^{3}\left\vert P_{\ell }(\cos d(x_{3},x_{4}))\right\vert \left\vert P_{\ell
}(\cos d(x_{1},x_{3}))\right\vert ^{4}\mu (dx_{1})\mu (dx_{2})\mu (dx_{3})\mu
(dx_{4}).
	\end{split}
	\end{equation}
Let us focus on the first term of (\ref{n1}). It is simple to check that, for any $x_{2}\in B_{r_{\ell} }$%
	\[
	\int_{B_{r_{\ell} }}\left\vert P_{\ell }(\cos d(x_{1},x_{2}))\right\vert
	^{4}\mu (dx_{1})\leq \int_{B_{2r_{\ell }}}\left\vert P_{\ell }(\cos
	d(N,x_{1}))\right\vert ^{4}\mu (dx_{1}),
	\]%
	\newline
	where $N$ denotes the North Pole (note the doubling of the radius in $%
	B_{2r_{\ell }}$
	).
	Since $|P_\ell(x)|\leq 1$, for $ x \in [0,1]$, we have that $\frac{1}{L}\int_{0}^{1} P_\ell (\cos \frac{\psi}{L})^q \frac{\psi}{L} \,d\psi=O(\frac{1}{\ell^2})$, $\forall q \geq 1$; then
	by Hilb's asymptotics (see Lemma \ref{hilbs}) %
	\begin{eqnarray*}
		\int_{B_{2r_{\ell} }}\left\vert P_{\ell }(\cos d(N,x_{1}))\right\vert ^{4}\mu
		(dx_{1}) &\leq &Const \times \frac{1}{\ell ^{2}}\int_{C}^{2\ell r_{\ell }}%
		\frac{1}{\psi }d\psi +O\bigg(\frac{1}{\ell^2}\bigg) \\
		&\leq &Const\times \frac{\log r_{\ell }\ell }{\ell ^{2}}+ O\bigg(\frac{1}{\ell^2}\bigg)
	\end{eqnarray*}%
	and similarly%
	\begin{eqnarray*}
		\int_{B_{r_{\ell}}}\left\vert P_{\ell }(\cos d(x_{3},x_{4}))\right\vert \mu
		(dx_{3}) &\leq &\int_{B_{2r_{\ell }}}\left\vert P_{\ell }(\cos
		d(N,x_{3}))\right\vert \mu (dx_{3}) \\
		&\leq &Const\times \frac{1}{\ell ^{2}}\int_{C}^{2\ell r_{\ell }}\sqrt{\psi }%
		d\psi +O\bigg(\frac{1}{\ell^2}\bigg) \\
		&\leq &Const\times \frac{r_{\ell }^{3/2}}{\sqrt{\ell }}+ O\bigg(\frac{1}{\ell^2}\bigg)\text{,}
	\end{eqnarray*}%
	\begin{eqnarray*}
		\int_{B_{r_{\ell}}}\left\vert P_{\ell }(\cos d(x_{2},x_{4}))\right\vert
		^{3}\mu (dx_{2}) &\leq &\int_{B_{2r_{\ell }}}\left\vert P_{\ell }(\cos
		d(N,x_{2}))\right\vert ^{3}\mu (dx_{2}) \\
		&\leq &Const\times \frac{1}{\ell ^{2}}\int_{C}^{2\ell r_{\ell }}\frac{1}{%
			\sqrt{\psi }}d\psi +O\bigg(\frac{1}{\ell^2}\bigg)\\
		&\leq &Const\times \frac{r_{\ell }^{1/2}}{\sqrt{\ell ^{3}}} +O\bigg(\frac{1}{\ell^2}\bigg)\text{,}
	\end{eqnarray*}%
	while obviously%
$$
	\int_{B_{r_{\ell}}}\mu (dx_{4})=O(r_{\ell }^{2})\text{.}
$$
	It follows that%
	\begin{equation*}
	\begin{split}
\int_{(B_{r_{\ell }})^{4}}&\left\vert P_{\ell }(\cos d(x_{1},x_{2}))\right\vert 
\left\vert P_{\ell }(\cos d(x_{1},x_{3}))^{3}\right\vert \left\vert P_{\ell
}(\cos d(x_{3},x_{4}))\right\vert \times \\& \times \left\vert P_{\ell }(\cos
d(x_{2},x_{4}))^{3}\right\vert \mu (dx_{1})\mu (dx_{2})\mu (dx_{3})\mu (dx_{4})
	=O\bigg(r_{\ell }^{4}\frac{\log r_{\ell }\ell }{\ell ^{4}}\bigg)\text{,}
		\end{split}
	\end{equation*}
	as needed. Equivalent computations give the same bound for the second term in (\ref{n1}).\\
As far as the term $A_2$ is concerned, we need to bound
\begin{equation}\label{n2}
	\begin{split}
\int_{(B_{r_{\ell }})^{4}}\left\vert P_{\ell }(\cos d(x_{1},x_{2}))^2\right\vert &
\left\vert P_{\ell }(\cos d(x_{1},x_{3}))^{2}\right\vert \left\vert P_{\ell
}(\cos d(x_{3},x_{4}))^2\right\vert \times \\& \times \left\vert P_{\ell }(\cos
d(x_{2},x_{4}))^2\right\vert \mu (dx_{1})\mu (dx_{2})\mu (dx_{3})\mu (dx_{4}).
	\end{split}
\end{equation}
The same strategy we have applied to $A_1$ leads (\ref{n2}) to be bounded by
$$
\int_{(B_{r_{\ell }})^{4}}\!\!\left\vert P_{\ell }(\cos d(x_{1},x_{2}))\right\vert
^{4}\left\vert P_{\ell }(\cos d(x_{3},x_{4}))^2\right\vert \left\vert P_{\ell
}(\cos d(x_{2},x_{4}))\right\vert ^{2}\mu (dx_{1})\mu (dx_{2})\mu (dx_{3})\mu
(dx_{4})
$$
$$
+\int_{(B_{r_{\ell }})^{4}}\!\!\!\!\left\vert P_{\ell }(\cos d(x_{1},x_{3}))\right\vert
^{4}\left\vert P_{\ell }(\cos d(x_{3},x_{4}))^2\right\vert \left\vert P_{\ell
}(\cos d(x_{2},x_{4}))\right\vert ^{2}\mu (dx_{1})\mu (dx_{2})\mu (dx_{3})\mu
(dx_{4})
$$
and since
\begin{equation*}
\begin{split}
\int_{B_{r_{\ell }}}P_{\ell }(\cos d(x_{3},x_{4}))^2 \,\mu(dx_3) &\leq \int_{B_{2r_\ell}} |P_\ell(\cos d (N,x_3)) |^2\,\mu(dx_3)
\\& \leq Const \times \dfrac{1}{\ell^2} \int_{C}^{2r_\ell L} \,d\psi+ O\bigg(\frac{1}{\ell^2}\bigg)\\& \leq Const \times \frac{r_\ell}{\ell}+O\bigg(\frac{1}{\ell^2}\bigg) ,
	\end{split}
\end{equation*}
we obtain that
$$A_2=O\bigg( r_\ell^2 \times \frac{r_\ell}{\ell} \frac{r_\ell}{\ell} \frac{\log r_\ell \ell}{\ell^2} \bigg)=O\bigg( \frac{r_\ell^4 \log r_\ell \ell}{\ell^4}\bigg)$$
and the conclusion of the lemma follows.
	\end{proof}

	\subsubsection{Proof of Theorem \ref{CLT}}\label{cltproof}
	From Lemma \ref{4cum} we conclude that
		\begin{equation}\label{eq}
	cum_4(\mathcal{M}_{\ell,r_\ell})=O\bigg( r_\ell^4 \log(r_\ell \ell)\bigg)
	\end{equation}
and, in view of Proposition \ref{lemma2.3},  the Fourth Moment Theorem (\cite{Peccati Nourdin}, Theorem 5.2.6) implies that 
$$d_W(\mathcal{M}_{\ell,r_\ell}, Z) \leq C\sqrt{\frac{cum_4(\mathcal{M}_{\ell,r_\ell})}{\Var(\mathcal{M}_{\ell,r_\ell})^2}}=O\bigg( \frac{1}{\sqrt{\log r_\ell \ell}} \bigg),$$
where $Z \sim \mathcal{N}(0,1)$ and $C$ is an explicit constant. Defining $$\tilde{\mathcal{M}}_{\ell,r_\ell}:= \dfrac{\mathcal{M}_{\ell,r_\ell}}{\sqrt{\Var(\mathcal{M}_{\ell,r_\ell})}} \mbox{, and } \tilde{\mathcal{Z}}_{\ell,r_\ell}:= \dfrac{\mathcal{Z}_{\ell,r_\ell}}{\sqrt{\Var(\mathcal{Z}_{\ell,r_\ell})}},$$ it follows that, as $\ell \rightarrow \infty$,
	$$d_{W}(\tilde{\mathcal{Z}}_{\ell,r_\ell}, Z) \leq d_W({\tilde{\mathcal{M}}}_{\ell,r_\ell},Z)+\sqrt{{E}[{\tilde{\mathcal{Z}}}_{\ell,r_\ell}-{\tilde{\mathcal{M}}}_{\ell,r_\ell}]^2}=O\bigg(\frac{1}{\sqrt{\log r_\ell \ell}}\bigg).$$

\section{Further Result: Correlation between $\mathcal{Z}_{\ell,r_{\ell}}$ and $\mathcal{Z}(T_\ell) $ (proof of Proposition \ref{correlation1})}\label{Section-further}
As we have already said in the introduction, contrary to the 2-dimensional torus, the nodal length on the total sphere and the one on its subregions are not correlated; indeed we prove here Proposition \ref{correlation1}. Before doing that, we compute the covariance between $\mathcal{Z}_{\ell,r_{\ell}}$ and $ \mathcal{Z}(T_\ell)$ in the lemma here below.
\begin{lem}\label{co} The covariance between $\mathcal{Z}_{\ell,r_{\ell}}$ and $ \mathcal{Z}(T_\ell)$ is given by
	$$\Cov(\mathcal{Z}_{\ell,r_{\ell}}, \mathcal{Z}(T_\ell))=\dfrac{|B_{r_{\ell}}|}{|{S}^2|} \Var(\mathcal{Z}(T_\ell)).$$
\end{lem}
\begin{proof} The proof of this lemma follows from the field's rotation invariance. Indeed, let consider $B_r$ the ball of radius $r$, for any $r>0$; we shall write the covariance as
	\begin{equation}
	\begin{split}
	{E}[\mathcal{Z}_{\ell,r}\cdot \mathcal{Z}(T_\ell) ]&=E\bigg[\int_{{S}^2 } ||\nabla(T_\ell(x))|| \delta(T_\ell(x)) \,dx \int_{B_{r}} ||\nabla T_\ell(y)|| \delta(T_\ell(y)) \,dy\bigg] \\&
	=\int_{{S}^2\times B_{r}} {E}[||\nabla T_\ell(x)|| ||\nabla T_\ell(y)|| \delta(T_\ell(x)) \delta(T_\ell(y))] \,dx dy\\&
	=\int_{{S}^2\times B_{r}} \tilde{K}_{\ell}(x,y) \, dxdy=|B_{r}| \int_{{S}^2} \tilde{K}_{\ell}(N,y) \, dy.
	\end{split}
	\end{equation}
	Then, taking $r=r_\ell$ and $r=\pi$, $B_r=B_{r_\ell}$ and $B_r=S^2$, respectively, we get	\begin{equation}
	{E}[\mathcal{Z}_{\ell,r_\ell}\cdot \mathcal{Z}(T_\ell) ]=|B_{r_\ell}| \int_{{S}^2} \tilde{K}_{\ell}(N,y) \, dy
	\end{equation}
	and 
	\begin{equation}
	{\Var}[ \mathcal{Z}(T_\ell) ]=|S^2| \int_{{S}^2} \tilde{K}_{\ell}(N,y) \, dy
	\end{equation}
	and the conclusion of the lemma follows.
\end{proof}
\begin{proof}[Proof of Proposition \ref{correlation1}]
	By definition, the correlation is
	\begin{equation}
	\begin{split}
	\Corr(\mathcal{Z}_{\ell,r_{\ell}};\mathcal{Z}(T_\ell))= \dfrac{\Cov(\mathcal{Z}_{\ell,r_{\ell}}; \mathcal{Z}(T_\ell))}{\sqrt{\Var(\mathcal{Z}_{\ell,r_{\ell}})} \sqrt{\Var(\mathcal{Z}(T_\ell))}}
	\end{split}
	\end{equation}
	and Lemma \ref{co} implies that
	\begin{equation}
	\begin{split}
	\Corr(\mathcal{Z}_{\ell,r_{\ell}};\mathcal{Z}(T_\ell))&=\dfrac{|B_{r_{\ell}}|}{|{S}^2|} \dfrac{\sqrt{\Var(\mathcal{Z}(T_\ell))}}{\sqrt{\Var(\mathcal{Z}_{\ell,r_{\ell}})}}=\dfrac{2\pi(1-\cos r_{\ell})}{4\pi} \dfrac{\sqrt{\Var(\mathcal{Z}(T_\ell))}}{\sqrt{\Var(\mathcal{Z}_{\ell,r_{\ell}})}}\\&=
	\dfrac{(1-\cos r_{\ell})}{2} \dfrac{\sqrt{\Var(\mathcal{Z}(T_\ell))}}{\sqrt{\Var(\mathcal{Z}_{\ell,r_{\ell}})}};
	\end{split}
	\end{equation}
	in view of Theorem \ref{varianza nodal} and (\ref{varianza vecchia}), it results that
	\begin{equation}
	\begin{split}
	\Corr(\mathcal{Z}_{\ell;r_{\ell}}; \mathcal{Z}(T_\ell))&=\dfrac{1-\cos r_{\ell}}{2} \sqrt{\dfrac{\frac{1}{32} \log {\ell} +O(1)}{\frac{r_{\ell}^2}{256} \log {\ell}r_{\ell}+O(r_{\ell}^2)}} = \dfrac{1-\cos r_{\ell}}{2r_{\ell}} \sqrt{\dfrac{ \log {\ell}}{\log(r_{\ell}  \ell )} +O(1)} \sqrt{8}\\&
	=\dfrac{1-\cos r_{\ell}}{2r_{\ell}^2} \sqrt{r_{\ell}^2 \frac{\log {\ell}}{\log r_{\ell}  \ell }+O(r_{\ell}^2)} \sqrt{8} =O\bigg(\sqrt{{r_\ell}^2 \frac{\log \ell}{\log r_\ell \ell} }\bigg).
	\end{split}
	\end{equation}
	Now, to prove that this quantity goes to zero, we note that, either $r_\ell \geq \frac{1}{\sqrt{\ell}}$, then 
	$$\bigg|r_\ell^2 \frac{\log \ell}{\log r_\ell \ell} \bigg| \leq \bigg|r_\ell^2 \frac{\log \ell}{\log \sqrt{r_\ell}} \bigg|= 2 r_\ell^2$$ which goes to zero because $r_\ell \rightarrow 0$; or if $r_\ell \leq \ell^{-1/2}$, since $r_\ell \ell \rightarrow +\infty$, we can bound $\log r_\ell \ell$ from below for $\ell$ large and get $$\bigg|r_\ell^2 \frac{\log \ell}{\log r_\ell \ell}\bigg|=O(r_\ell^2 \log \ell)=O\left(\frac{1}{\ell} \log \ell \right)=o(1).$$
\end{proof}
	\section{Technical tools} \label{App2point}
	In this section we collect some results exploited in the previous computations.\\
For the purpose of the present paper, let us note the following result.
\begin{lem}\label{lemmaK} For $1 <\psi<r_\ell \ell,$ as $\ell \rightarrow \infty$,
	$$K_\ell-\frac{1}{4}= \frac{1}{2\pi} \frac{\sin(2\psi)}{\psi}+ \frac{P_1(\psi)}{\psi^2} + O\left(\frac{1}{\psi^3}\right), $$
	where $P_1(\psi)$ is the trigonometric polynomial given by
	\begin{equation}\label{PK}
	P_1(\psi)=\dfrac{1}{256\pi^2} + \dfrac{9}{32 \pi}\cos(2\psi)+\frac{27}{64\pi^2}\sin(2\psi)-\frac{75}{256\pi^2}\cos (4\psi).
	\end{equation}
\end{lem}

\begin{proof}
	Let us consider the expansion in (\ref{56})
	holding uniformly for $C<\psi<\dfrac{\pi L}{2}$.
	In the regime $[1,r_\ell \ell]$, $r_\ell \ell=o(L)$ and the terms $\sin (\psi/L)$ appearing in all the denominators can be replaced by 
	\begin{equation}\label{taylor sin}
	\sin \frac{\psi}{L}+O \bigg( \frac{\psi^3}{L^3} \bigg).
	\end{equation}
	Hence, we have
	\begin{equation}
	\begin{split}
	K_\ell(\psi)-\frac{1}{4}&= \bigg[ \frac{\sin 2\psi}{2\pi \ell}+ \frac{1}{256 \pi^2 \ell \psi}+ \frac{9\cos 2\psi}{32 \pi \psi \ell }+ \frac{\frac{27}{64}\sin 2\psi-\frac{75}{256}\cos (4\psi)}{\pi^2 \ell \psi \sin (\psi/L)} \bigg] \frac{1}{\frac{\psi}{L}+ O(\psi^3/L^3)}+\\&+O\bigg(\frac{1}{\psi^3}+\frac{1}{\ell \psi}\bigg)\\&
	=\bigg[ \frac{\sin 2\psi}{2\pi \psi}+ \frac{1}{256 \pi^2 \psi^2}+ \frac{9\cos 2\psi}{32 \pi \psi^2 }+ \frac{\frac{27}{64}\sin 2\psi-\frac{75}{256}\cos (4\psi)}{\pi^2 \psi^2 \sin (\psi/L)} \bigg] + O\bigg(\dfrac{1}{\psi^3}\bigg)
	\end{split}
	\end{equation}
	Denoting $P_1(\psi)$ the trigonometric polynomial given in (\ref{PK}) the conclusion of the lemma follows.
\end{proof}


\begin{lem}\label{lemmaJ} For $1<\psi <r_\ell \ell$, as $\ell \rightarrow \infty$,
	$$\mathcal{J}_\ell (\psi;4)= L \frac{P_2(\psi)}{\psi^2}+O\left(\frac{1}{\psi^3}\right),$$
	where the trigonometric polynomial $P_2(\psi)$ is
	\begin{equation}\label{PJ}
	P_2(\psi)= \dfrac{1}{64}+\frac{5}{64} \cos 4\psi-\frac{3}{16} \sin(2\psi).
	\end{equation}
\end{lem}

\begin{proof}
	Similarly to the proof of Lemma \ref{lemmaK}, we substitute $\sin(\psi/L)$ with its Taylor expansion (\ref{taylor sin}) in equation (\ref{3.3}), holding for $C<\psi<L \frac{\pi}{2}$, and defining $P_2(\psi)$ as in (\ref{PJ}) the thesis follows.
\end{proof}
Other useful results for our computations are given by the following lemmas.
\begin{lem}\label{lemma utile1} As $x \rightarrow \infty$,
	$$\int_{1}^{x} \frac{1}{\psi^2} \,d\psi=O(1).$$
\end{lem}

\begin{lem}\label{lemmaP}
	Let $P(\psi)=a_0+a_1 \cos \psi+\dots+a_m \cos(m\psi)+b_1 \sin(\psi)+\dots+b_m \sin(m\psi)$ a general trigonometric polynomial. Then, as $x \rightarrow +\infty$, $$\int_{1}^{x} \frac{P(\psi)}{\psi} \,d\psi=a_0 \log (x)+O(1).$$
\end{lem}

\begin{proof}
	We have that $$\int_{1}^{x} \frac{P(\psi)}{\psi} \,d\psi= \int_{1}^{x} \frac{a_0+a_1 \cos \psi+\dots+a_m \cos(m\psi)+b_1 \sin(\psi)+\dots+b_m \sin(m\psi)}{\psi} \,d\psi.$$
	Let us focus, for example, on $$\int_{1}^{x} a_1 \frac{\cos(\psi)}{\psi} \,d\psi.$$ Integrating by parts, it becomes
	$$ a_1\bigg[ \frac{\sin(\psi)}{\psi} -\int_{1}^{x} \frac{\sin\psi}{\psi^2} \,d\psi\bigg]_{1}^{x}$$
	and thanks to the Lemma \ref{lemma utile1} and to the fact that the function $\sin \psi$ is bounded, it is $O(1)$, as $\ell \rightarrow \infty$. In the same way, it is possible to see that, as $\ell \rightarrow \infty$,
	$$ \int_{1}^{x} \frac{a_2 \cos 2 \psi+\dots+a_m \cos(m\psi)+b_1 \sin(\psi)+\dots+b_m \sin(m\psi)}{\psi} \,d\psi=O(1)$$ and hence the leading term of $\int_{1}^{x} \frac{P(\psi)}{\psi} \,d\psi$ is given by
	$$\int_{1}^{x} \frac{a_0}{\psi} \,d\psi=a_0 \log(x).$$
\end{proof}

As a consequence of Lemma \ref{lemmaP}, we get the following corallary.
\begin{cor}\label{corollary}
	As $\ell \rightarrow \infty$, 
	\begin{equation}\label{lemmaKpsi}
	\int_{1}^{r_\ell \ell} \left(K_\ell(\psi)-\frac{1}{4}\right) \psi d\psi= \frac{1}{256 \pi^2} \log(r_\ell \ell)+O(1)
	\end{equation}
and
	\begin{equation}\label{lemma Jpsi}
	\frac{1}{L}\int_{1}^{r_\ell \ell} \mathcal{J}_\ell(\psi)\psi d\psi= \frac{1}{64} \log(r_\ell \ell)+O(1).
	\end{equation}
\end{cor}

\begin{lem}[Hilb's Asymptotics (formula (8.21.17) on page 197 in \cite{szego})]\label{hilbs}
	\begin{equation}
	P_{\ell}(\cos \phi)= \bigg(\frac{\phi}{\sin \phi}\bigg)^{1/2} J_0((\ell+1/2)\phi)+\delta(\phi),
	\end{equation}
	uniformly for $0\leq \phi \leq \pi/2$, where $J_0$ is the Bessel function of order 0, defined as
	$J_0(x)=\sum_{k=0}^{\infty} \dfrac{(-1)^kx^{2k}}{2^{2k}(k!)^2}$, and the error term is
	$$\delta(\phi) \ll  \begin{cases} \phi^{1/2} O(\ell^{-3/2}), & C\ell^{-1}<\phi<\pi/2 \\ \phi^{2}O(1), & 0<\phi<C\ell^{-1},\end{cases}$$
	where $C>0$ is any constant and the constants involved in the $``O"$-notation depend on $C$ only.
\end{lem}
In particular, for $\theta \in [0,\pi/2],$
\begin{equation}\label{hilb2}
P_\ell(\cos\theta) \ll \frac{1}{\sqrt{\ell \theta}}.
\end{equation}
Actually, changing variable $\Psi=L \theta$, with $L=\ell+\frac{1}{2}$, we have that
$$P_\ell \bigg(\cos \bigg(\frac{\psi}{\ell+1/2} \bigg) \bigg) \sim J_0(\psi) $$ and
$$J_0(\psi)= \sqrt{\frac{2}{\pi}} \frac{\cos(\psi-\pi/4)}{\sqrt{\psi}}+O\bigg(\frac{1}{\psi^{3/2}}\bigg)$$
(see also \cite{M e W 2012}).


Lemma \ref{hilbs} implies (\ref{Pazero}) and the following result can be easily seen.

\begin{lem}\label{lemmaP4}
 For $1<\psi<r_\ell \ell$, as $\ell \rightarrow \infty$,
	$$
	P_{\ell}(\cos (\psi/L))^4=\frac{P_3(\psi)}{\psi^2}+O\bigg(\dfrac{1}{\psi^3}\bigg),$$
	where
	$$P_3(\psi)=\frac{3}{2\pi^2}-\frac{2}{\pi^2}\sin(2\psi)-\frac{1}{2\pi^2} \cos(4\psi).$$
\end{lem}

\section*{Acknowledgements}
The author would like to thank Domenico Marinucci and Igor Wigman for the proposal of the topic, for all the useful suggestions and for all the discussions and the insightful remarks. Most of the research was done in the department of mathematics of King's College of London and in the University of Rome Tor Vergata, to which the author is grateful for the warm hospitality. Many thanks to an anonimus referee for the useful suggestions and remarks. Finally, thanks to Valentina Cammarota for some suggestions. The author was financially supported by the GSSI, the UMI with the Grant for Visiting student and the German Research Foundation (DFG) via RTG 2131.


\newpage

\section*{Supplementary Material}
\appendix
\section{Background Material}
\subsection{Wiener Chaos}

In this part we recall the notion of Wiener chaos mentioned in the introduction. For a complete discussion see \cite{Peccati Nourdin}, Chap. 2.2.
Let us consider the sequence $\{ H_q \}_{q \in N}$ of Hermite polynomials on $R$, defined as follows
$$H_0 = 1$$
$$H_q(t) = tH_{q-1}(t) - H_{q-1}^\prime (t), \mbox{ } q \geq 1.$$ 
It is useful to recall the following property: let $Z_1, Z_2$ jointly
Gaussian; then, for all $q_1, q_2 \geq 0$ 
\begin{equation}\label{propertyH}
E[H_{q_1}(Z_1)H_{q_2}(Z_2)] = q_1!\delta_{q_1}^{q_2} E[Z_1Z_2].
\end{equation}
Now, we recall that the family $\mathbf{H}=\{H_q, q \geq 0\}$ is a complete orthogonal system in the space of square integrable functions $L^2(\gamma)$, where $\gamma$ denotes the standard Gaussian density on R. We define the space $\chi$ to be the closure in $L^2(P)$ of all real finite linear combinations of random variables $\xi$ of the form $\xi= z a_{\ell m}+ \bar{z}(-1)^{\ell} a_{\ell, -m}$, $z \in C$ and $a_{\ell m}$ independent Gaussian random variables  with the condition $\bar{a}_{\ell m}=(-1)^\ell a_{\ell, -m}$. The space $\chi$ is a real centered Gaussian Hilbert subspace of $L^2(P)$. \\
We define the space of constants $C_0:=R \subset L^2(P)$ and for $q \geq 1$ an integer, the $q-$th Wiener chaos $C_q$ associated with $\chi$ is the closure of all real finite linear combinations of random variables of the type 
$$H_{p_1}(\xi_1)H_{p_2}(\xi_2)\cdots H_{p_k}(\xi_k)$$
$k \geq 1$, where the integers $p_1,\dots,p_k \geq 0$ are such that $p_1+\dots+p_k=q$ and $(\xi_1,\dots,\xi_k)$ is a standard real Gaussian vector extracted from $\chi$. It is possible to prove that $C_q \perp C_m$ in $L^2(P)$ for $q \ne m$ and that
$$L^2(\Omega, \sigma(\chi), P)=\bigoplus_{q=0}^{\infty} C_q.$$
Then, every real-valued functional $F$ of $\chi$ can be (uniquely) represented as a series, converging in $L^2$, of the form $$F=\sum_{q=0}^{\infty} Proj(F|C_q)$$ where the $Proj(F|C_q)$ is the projection of $F$ onto $C_q$ (in particular $Proj(F|C_0)=E[F]$).  

\subsubsection{Chaotic expansion for nodal lengths}
In the same lines of the case of the sphere (see \cite{MRW}, \cite{Rossi2}) an integral representation for the nodal length $\mathcal{Z}_{\ell,r_\ell}$ can be given by
$$\mathcal{Z}_{\ell,r_\ell}= \int_{B_{r_\ell}} \delta_0(T_\ell(x)) ||\nabla T_\ell(x)||\,dx,$$ where $\delta_0$ denotes the dirac delta function and $||\cdot||$ the standard Euclidean norm in $R^2$. This representation can be shown to hold almost surely in $\Omega$ and it is shown hold in $L^2(\Omega)$ (see \cite{MRW}). The $L^2$ expansion of nodal lengths takes the form (see \cite{MRW}, \cite{MPRW} and \cite{Rossi2})
\begin{equation*}
\begin{split}
\mathcal{Z}_{\ell,r_\ell}-E[\mathcal{Z}_{\ell,r_\ell}]= &\sqrt{\frac{\ell(\ell+1)}{2}} \sum_{q=2}^{\infty} \sum_{u=0}^{q} \sum_{k=0}^{u} \frac{\alpha_{k,u-k} \beta_{q-u}}{k!(u-k)!(q-u)!}  \\& \times \int_{B_{r_\ell}} H_{q-u}(T_\ell(x)) H_{k}\bigg(\frac{\partial_{1;x}T_\ell(x)}{\sqrt{\ell(\ell+1)/2}}\bigg) H_{u-k}\bigg(\frac{\partial_{2;x}T_\ell(x)}{\sqrt{\ell(\ell+1)/2}}\bigg)\, dx\\&
=\sum_{q=2}^{\infty} \int_{B_{r_\ell} } \Psi_\ell(x;q)\,dx,
\end{split}
\end{equation*}
where 
\begin{equation}
\begin{split}
\Psi_\ell(x;q)=\sqrt{\frac{\ell(\ell+1)}{2}} \sum_{u=0}^{q} \sum_{k=0}^{u} \frac{\alpha_{k,u-k} \beta_{q-u}}{k!(u-k)!(q-u)!} &H_{q-u}(T_\ell(x)) H_{k}\bigg(\frac{\partial_{1;x}T_\ell(x)}{\sqrt{\ell(\ell+1)/2}}\bigg) \\& \times H_{u-k}\bigg(\frac{\partial_{2;x}T_\ell(x)}{\sqrt{\ell(\ell+1)/2}}\bigg)\, dx.
\end{split}
\end{equation}
In spherical coordinates $(\theta,\varphi)$ and for $x=(\theta_x,\varphi_x)$,
$$\partial_{1;x}=\frac{\partial}{d \theta}\bigg|_{\theta=\theta_x}\mbox{ ,} \mbox{ }  \partial_{2;x}=\frac{1}{\sin \theta} \frac{\partial}{d \varphi} \bigg|_{\theta=\theta_x, \varphi=\varphi_x}.$$
In particular, denoting as $\tilde{\mathcal{Z}}_{\ell,r_\ell}=\dfrac{\mathcal{Z}_{\ell,r_\ell}-E[\mathcal{Z}_{\ell,r_\ell}]}{\sqrt{\Var({\mathcal{Z}_{\ell,r_\ell}})}}$, the projection of the nodal length on the fourth-order chaos has the expression
$$Proj(\tilde{\mathcal{Z}}_{\ell,r_\ell}|C_4)= \int_{B_{r_\ell}} \Psi_\ell(x;4)\,dx$$
\begin{equation*}
\begin{split}
&= \sqrt{\frac{\ell(\ell+1)}{2}} \bigg\{  \frac{\alpha_{0,0} \beta_{4}}{4!} \int_{B_{r_\ell} } H_{4}(T_\ell(x)) \,dx +  \frac{\alpha_{2,0} \beta_{2}}{2!2!} \int_{B_{r_\ell} } H_{2}(T_\ell(x)) H_{2}\bigg(\frac{\partial_{1;x}T_\ell(x)}{\sqrt{\ell(\ell+1)/2}}\bigg) \,dx \\& +
\frac{\alpha_{4,0} \beta_{4}}{4!} \int_{B_{r_\ell}} \!H_{4}\bigg(\frac{\partial_{1;x}T_\ell(x)}{\sqrt{\ell(\ell+1)/2}}\bigg) \,dx+  \frac{\alpha_{2,2} \beta_{0}}{2!2!} \int_{B_{r_\ell}} \!H_{2}\bigg(\frac{\partial_{1;x}T_\ell(x)}{\sqrt{\ell(\ell+1)/2}} \bigg) H_{2}\bigg(\frac{\partial_{2;x}T_\ell(x)}{\sqrt{\ell(\ell+1)/2}}\bigg) \, dx\\&
+   \frac{\alpha_{0,2} \beta_{2}}{2!2!} \int_{B_{r_\ell}} H_{2} (T_\ell(x)) H_2 \bigg(\frac{\partial_{2;x}T_\ell(x)}{\sqrt{\ell(\ell+1)/2}}\bigg) \,dx+  \frac{\alpha_{0,4} \beta_{0}}{4!} \int_{B_{r_\ell}} H_{4}\bigg(\frac{\partial_{2;x}T_\ell(x)}{\sqrt{\ell(\ell+1)/2}}\bigg) \, dx \bigg\}.
\end{split}
\end{equation*}

\subsection{On the Fourth cumulant of the fourth chaotic projection}
Let us consider in this section the following lemma proved in the paper.
\begin{lem}\label{4cumS}
	Let $h_{\ell,r_\ell,4}$ defined as 
	\begin{equation}\label{h4S}
	h_{\ell,r_\ell,4}:=\int_{B_{r_{\ell}}}H_4(T_\ell(x)) \, dx,
	\end{equation}then, as $\ell \rightarrow \infty,$
	\begin{equation}\label{unaequa}
	\begin{split}
	cum_4\{h_{\ell,r_\ell,4}\}
	=O\bigg( \dfrac{r_\ell^4}{\ell^4} \log r_\ell \ell \bigg).
	\end{split}
	\end{equation}
\end{lem}
We stated at the beginning of the proof of Lemma 5.4 that, to bound the fourth cumulant of  $h_{\ell,r_\ell,4}$, it is sufficient to study the two integrals: 
$$A_1=\int_{B_{r_{\ell}}^4} P_{\ell}(\langle x_1,x_2\rangle)P_{\ell}(\langle x_1,x_3\rangle)^3  P_{\ell}(\langle x_3,x_4\rangle)P_{\ell}(\langle x_2,x_4 \rangle)^3\, \mu (dx_{1})\mu (dx_{2})\mu (dx_{3})\mu (dx_{4})$$ and
$$A_2=\int_{B_{r_{\ell}}^4} P_{\ell}(\langle x_1,x_2\rangle)^2 P_{\ell}(\langle x_1,x_3\rangle)^2  P_{\ell}(\langle x_3,x_4\rangle)^2P_{\ell}(\langle x_2,x_4 \rangle)^2\, \mu (dx_{1})\mu (dx_{2})\mu (dx_{3})\mu (dx_{4}),$$
where $\mu (dx_{i})$ denotes Lebesgue measure on the sphere.
To see that, we can follow exactly the argument in \cite{M e W 2012}, which we report for completeness. Hence, we recall that a diagram is a graph with $(\alpha_1+\dots+\alpha_p)$ vertexes labelled by $1,\dots,p$, such that each vertex has degree 1. The set of all such graphs $\gamma$ is denoted by $\Gamma(\alpha_1,\dots,\alpha_p)$. We denote by $\Gamma_C (\alpha_1,\dots,\alpha_p)$ the graphs which are connected.\\
Given a diagram $\gamma$, let $\eta(\gamma)=\eta_{ij}(\gamma) \in Z^{\binom{p}{2}}$ the vector whose $\binom{p}{2}$ elements $\eta_{i,j}(\gamma)$ $(i <j)$ are the number of edges between $i$ and $j$ in the graph $\gamma.$ The vector $\eta$ satisfies $\sum_{ij} \eta_{i,j}=2q.$ The following lemma is proved in \cite{M e W 2012}.
\begin{lem}\label{rid}[\cite{M e W 2012}, Lemma 2.1]
	Let $\gamma \in \Gamma_C(q,q,q,q)$ with arbitrary $q\geq1$, and $\eta=\eta(\gamma)$. Let $e=(i,j)$ any edge in $\gamma$ and $e^\prime=(i^\prime,j^\prime)$ the unique edge with vertexes disjoint with $e$, so that $\{ i,j,i^\prime,j^\prime \}=\{ 1,2,3,4 \}$. Then $\eta_e=\eta_{e^\prime}$.
\end{lem}
It is shown in \cite{M e W 2012} that the fourth cumulant can be computed by
$$cum_4\bigg[\int_{B_{r_\ell}}H_4(T_\ell(x)) \, dx \bigg]=\sum_{\gamma \in \Gamma_c(4,4,4,4)} M(\eta(\gamma))$$ where for a vector $\eta \in Z^6_{\geq 0}$,
$$M(\eta)= \int_{B_{r_{\ell}}^4} \prod_{i<j } P_{\ell}(\langle x_i,y_j\rangle)^{\eta_{i,j}} \, \mu (dx),$$ where $\mu (dx)=\mu (dx_{1})\mu (dx_{2})\mu (dx_{3})\mu (dx_{4})$.
Now, we use Lemma \ref{rid} and the Cauchy-Schwartz inequality to reduce the number of different angles; when we apply the latter inequality, it is advantageous to pair up angles corresponding to disjoint edges in the diagram. In the end all the configuarations can be bounded by ones where $\eta$ has one of the following two shapes
$$\eta=(2,2,2,2,0,0)$$ or $$\eta=(1,3,1,3,0,0).$$
Then the proof of Lemma \ref{4cumS} reduces to the control of the two integrals $A_1$ and $A_2$.
Another way to prove this reduction follows by \cite{PT}, Proposition 11.2, where it is proved that it is sufficient to bound only the terms corresponding to circular diagrams (i.e. diagrams,
all of whose rows are linked with precisely two other rows) to establish the CLT.

\section{Technical tools}

\subsection{$L^2$ approximation for nodal lengths}
Following the same idea and notation in \cite{MRW}, we define
\begin{equation}\label{Zapprox}
\mathcal{Z}_{\ell,r_\ell;\eps}:= \int_{B_{r_\ell}} ||\nabla f_{\ell}(x)||\chi_\eps(T_\ell(x)) \,dx. 
\end{equation}

We show the $L^2$ convergence of the nodal length in the lemma here below.

\begin{lem}\label{L2approx} Let $\mathcal{Z}_{\ell,r_\ell;\eps}$ be defined as in (\ref{Zapprox}), we have that, as $\eps \rightarrow 0$,
	\begin{equation}
	\lim_{\eps \rightarrow 0} E[|\mathcal{Z}_{\ell,r_\ell;\eps}-\mathcal{Z}_{\ell,r_\ell}|^2]=0.
	\end{equation}
\end{lem}
\begin{proof} This argument follows closely \cite{MRW} and it is included for completeness. Hence, the nodal length is defined almost-surely by
	$$\lim_{\eps \rightarrow 0} \int_{B_{r_\ell}} \chi_\eps(T_\ell(x)) ||\nabla T_\ell(x)||\,dx$$
	and from the standard argument (\cite{RW}, Lemma 3.1) the almost-sure convergence follows. Indeed, since $\chi_\eps(\cdot)= \frac{1}{2\eps}1_{[-\eps,\eps]}(\cdot)$ is integrable and $T_\ell$ is smooth, we have thanks to the co-area formula (\cite{Adler e Taylor}, p.169)
	$$\int_{B_{r_\ell}} \chi_\eps(T_\ell(x)) ||\nabla T_\ell(x)||\,dx= \int_{R} \bigg\{ \int_{T_\ell^{-1}(s) \cap B_{r_\ell}} \chi_\eps(T_\ell(x) ) \, dx\bigg\} \,ds.$$
	Since
	$$\chi_\eps(T_\ell(x))=\begin{cases}0 & \mbox{ for } x: T_\ell(x) >\eps \\ \frac{1}{2\eps} & \mbox{ for } x : T_\ell(x) \leq \eps
	\end{cases}$$
	and the function $s \rightarrow \Vol[T_\ell(s)^{-1} \cap B_{r_\ell}]$ is continuous for regular (Morse) functions, we obtain
	$$\int_{R} \bigg\{ \int_{T_\ell^{-1}(s) \cap B_{r_\ell}} \chi_\eps(T_\ell(x) )\, dx \bigg\} \,ds= \frac{1}{2\eps} \int_{-\eps}^{\eps} \Vol \left[T_\ell^{-1}(s) \cap B_{r_\ell}\right]\,ds \rightarrow \Vol \left[T_\ell^{-1}(0) \cap B_{r_\ell} \right],$$
	as $\eps\rightarrow 0.$ We now show that the convergence occurs also in the $L^2$ sense. To this aim, since the convergence holds almost surely, it is sufficient to prove that $$\lim_{\eps \rightarrow 0} E[\mathcal{Z}_{\ell,r_\ell,\eps}^2] =E[\mathcal{Z}_{\ell,r_\ell}^2].$$
	Note that,
	\begin{equation}
	\begin{split}
	E[\mathcal{Z}_{\ell,r_\ell,\eps}^2]&= E \bigg[ \bigg\{  \int_{B_{r_{\ell}}} \{ \chi_\eps(T_\ell(x)) ||\nabla T_\ell(x)|| \} \,dx \bigg\}^2 \bigg]\\&
	= E \bigg[ \bigg\{  \int_{R} \int_{\{x \in B_{r_\ell }:T_\ell(x)=u\}} \{ \chi_\eps(T_\ell(x)) ||\nabla T_\ell(x)|| \} \,dx \bigg\}^2 \bigg]\\&
	= E \bigg[ \bigg\{  \int_{R} \mathcal{Z}_{\ell,r_\ell}(u) \chi_\eps(T_\ell(u)) \,du \bigg\}^2 \bigg].
	\end{split}
	\end{equation}
	The application $u \rightarrow E[\{ \mathcal{Z}_{\ell,r_\ell}(u) \}^2]$, where $\mathcal{Z}_{\ell,r_\ell}(u)=\len (\{ x \in S^2 \cap B_{r_\ell}: T_\ell(x)=u  \})$, is continuous, where 
	\begin{equation}
	\begin{split}
	&E[\mathcal{Z}_{\ell,r_\ell}^2(u)]\!=\!\int_{B_{r_\ell} \times B_{r_\ell}}\!\!\!\!\!\!\!\! E\left[||\nabla T_\ell(x_1)|| ||\nabla T_\ell(x_2)|| | T_\ell(x_1)\!=\!u, T_\ell(x_2)\!=\!u\right] \phi_{T_\ell(x_1), T_\ell(x_2)} (u,u) \,dx_1 dx_2\\&
	=\!8\pi^2 \int_{0}^{2r_\ell} \!\!\! E\left[||\nabla T_\ell(N)|| ||\nabla T_\ell(y(\rho))|| | T_\ell(N)\!=\!u, T_\ell(y(\rho))\!=\!u\right] \phi_{T_\ell(N), T_\ell(y(\rho))} (u,u) W_{r_\ell}(\rho) \,d\rho.
	\end{split}
	\end{equation}
	To check the continuity, it is enough to show that the Dominated Convergence Theorem holds; we first note that 
	$$ \phi_{T_\ell(N), T_\ell(y(\rho))} (u,u) W_{r_\ell}(\rho) \leq  \phi_{T_\ell(N), T_\ell(y(\rho))} (0,0) W_{r_\ell}(\rho)= \dfrac{1}{2\pi \sqrt{1-P_\ell(\cos \theta)^2}}W_{r_{\ell}} (\rho)$$ which is $O(1)$ uniformly in $\rho$
	since $W_{r_\ell}(\rho) \sim r_\ell^3 \tilde{W}_1(\frac{\rho}{r_\ell}) $ (where $\sim$ means that $\lim_{\rho \rightarrow 0} \frac{W_{r_\ell}(\rho)}{r_\ell^3 \tilde{W}_1(\frac{\rho}{r_\ell})}=1 $) and 
	$$\tilde{W}_1(\frac{\rho}{r_\ell})= \frac{1}{8\pi^2} \frac{\rho}{r_\ell} \int_{0}^{2\pi} \tilde{\varphi}_x(\rho \cos \phi,\rho \sin \phi)\, d\phi dx$$ with $\tilde{\varphi}_x(\rho \cos \phi,\rho \sin \phi)$ a bounded function. On the other hand the evaluation of $$E[||\nabla T_\ell(N)|| ||\nabla T_\ell(y(\rho))|| | T_\ell(N)=u, T_\ell(y(\rho))=u]$$ is given in \cite{MRW} and it is seen to be uniformly bounded over $\rho$. Then, the Dominated Convergence Theorem holds. It follows that
	
	\begin{equation}
	\begin{split}
	E[\mathcal{Z}_{\ell,r_\ell}^2] &\leq \liminf_{\eps \rightarrow 0} E \bigg[ \bigg\{  \int_{B_{r_\ell}} \{\chi_\eps(T_\ell(x)) ||\nabla T_\ell(x)||  \} \,dx
	\bigg\}^2 \bigg]\\&
	= \liminf_{\eps \rightarrow 0} E[\mathcal{L}_{\ell,r_\ell;\eps}^2] \leq \limsup_{\eps \rightarrow 0} E[\mathcal{Z}_{\ell,r_\ell;\eps}^2]\\&
	= \limsup_{\eps \rightarrow 0} E \bigg[ \bigg\{  \int_{B_{r_\ell}} \{\chi_\eps(T_\ell(x)) ||\nabla T_\ell(x)||  \} \,dx 
	\bigg\}^2 \bigg]\\&
	= \limsup_{\eps \rightarrow 0} E \bigg[ \bigg\{ \int_{R} \{  \mathcal{Z}_{\ell,r_\ell}(u) \chi_\eps(u)     \} \,du \bigg\}^2\bigg] \\&
	\leq \limsup_{\eps \rightarrow 0} \int_{R}  E[  \mathcal{Z}_{\ell,r_\ell}^2(u)] \chi_\eps(u)   \,du= E[\mathcal{Z}_{\ell,r_\ell}^2].
	\end{split}
	\end{equation}
	
\end{proof}


\subsection{On the proof of Lemma 5.2 }

In this section we want to prove the following result.
\begin{lem}\label{A}
	$$\Cov(\mathcal{Z}_{\ell,r_\ell}, \mathcal{M}_{\ell,r_\ell})= \lim_{i \rightarrow \infty} \int_{0}^{2r_\ell \ell} \mathcal{J}_\ell (\psi,4) W^{\varphi_{{\ell}}^i}(\frac{\psi}{L}) \,d\psi,$$
	where
	$$\mathcal{\mathcal{J}}_{\ell}(\psi;4)=\bigg[ -\frac{1}{4} \sqrt{\dfrac{\ell(\ell+1)}{2}} \dfrac{1}{4!} \bigg] \times \dfrac{8\pi^2}{L} {E}\left[\Psi_{\ell}(\bar{x},4)H_4\left(T_\ell\left(y\left(\frac{\psi}{L}\right)\right)\right)\right],$$ where we wrote $\bar{x}=(0,0)$ for the North Pole and $y(\rho)=(0,\rho)$ for the points on the meridian where $\varphi = 0$.
\end{lem}

\begin{proof} To prove this lemma we can follow the same steps as in \cite{MRW}, proof of Theorem 1.2.. We report them for completeness.
	Let us define $$\Psi_\eps(x):= ||\nabla f_{\ell}(x)||\chi_\eps(T_\ell(x)), \mbox{ } \chi_\eps(\cdot)=\frac{1}{2\eps} 1_{[-\eps,\eps]}(\cdot).$$
	
	$\Psi_\eps(x)$ admits the $L^2(\Omega)$ expansion
	
	$$\Psi_{\eps}(x)=E[\Psi_\eps(x)]+ \sum_{q=2}^{\infty} \Psi_{\ell;\eps}(x;q);$$ 
	
	moreover, we established in Lemma \ref{L2approx} the $L^2(\Omega)$ convergence 
	$$\lim_{\eps \rightarrow 0} \int_{B_{r_\ell}} \Psi_{\ell}(x) \, dx= \lim_{\eps \rightarrow 0} \int_{B_{r_\ell}} ||\nabla f_{\ell}(x)||\chi_\eps(T_\ell(x)) \,dx = \mathcal{Z}_{\ell,r_\ell}.$$

	Note also that $\Psi_{\ell}(x), H_4(T_\ell(y))$ are both in $L^2(S^2 \times \Omega)$ and they are isotropic and thus
	\begin{equation}\label{nonso}
	\begin{split}
	\Cov(\mathcal{Z}_{\ell,r_\ell;\eps} ;\mathcal{M}_{\ell,r_\ell})&=-\dfrac{1}{4}\sqrt{\dfrac{\ell(\ell+1)}{2}}\dfrac{1}{4!}\Cov\bigg(\int_{B_{r_\ell}}  \Psi_\eps(x) \,dx, \int_{B_{r_\ell}} H_4(T_\ell(y)) \, dy\bigg)\\&
	=-\dfrac{1}{4}\sqrt{\dfrac{\ell(\ell+1)}{2}}\dfrac{1}{4!} {E} \left[\int_{B_{r_\ell}} \Psi_{\eps}(x) \,dx\int_{B_{r_\ell}}  H_4(T_\ell(y))\,dy\right]\\&
	=-\dfrac{1}{4}\sqrt{\dfrac{\ell(\ell+1)}{2}}\dfrac{1}{4!} \int_{B_{r_\ell}}\int_{B_{r_\ell}} {E}[\Psi_{\eps}(x)H_4(T_\ell(y))] \,dx dy\\&
	=-\dfrac{1}{4}\sqrt{\dfrac{\ell(\ell+1)}{2}}\dfrac{1}{4!}\int_{{S}^2 \times {S}^2} {E}\left[ \sum_{q=2}^{\infty} \Psi_{\ell;\eps}(x,q)H_4(T_\ell(y))\right] 1_{B_{r_\ell}}(x) 1_{B_{r_\ell}}(y) \, dx dy\\&
	=-\dfrac{1}{4}\sqrt{\dfrac{\ell(\ell+1)}{2}}\dfrac{1}{4!}\int_{{S}^2 \times {S}^2} {E}[ \Psi_{\ell;\eps}(x,4)H_4(T_\ell(y))] 1_{B_{r_\ell}}(x) 1_{B_{r_\ell}}(y) \, dx dy\\&
	=\lim_{i \rightarrow \infty} -\dfrac{1}{4}\sqrt{\dfrac{\ell(\ell+1)}{2}}\dfrac{1}{4!}\int_{{S}^2 \times {S}^2} {E}[  \Psi_{\ell;\eps}(x,4)H_4(T_\ell(y))] \varphi_\ell^i(x) \varphi_\ell^i(y) \, dx dy
	\end{split}
	\end{equation}
	in the last passage we exploited the $L^1(S^2)$ convergence of $\varphi_{{\ell}}^i$ to $1_{B_{r_\ell}}$. Indeed,
	\begin{equation}\label{A4}
	\begin{split}
	\bigg|&  \int_{{S}^2 \times {S}^2} {E}[  \Psi_{\ell;\eps}(x,4)H_4(T_\ell(y))] \varphi_\ell^i(x) \varphi_\ell^i(y)\, dx dy \\&- \int_{{S}^2 \times {S}^2} {E}[\Psi_{\eps}(x)H_4(T_\ell(y))] 1_{B_{r_\ell}}(x)1_{B_{r_\ell}}(y) \,dx dy \bigg|\\&
	\leq 	\int_{{S}^2 \times {S}^2} 	| {E}[  \Psi_{\ell;\eps}(x,4)H_4(T_\ell(y))]| | [\varphi_\ell^i(x) \varphi_\ell^i(y) -  1_{B_{r_\ell}}(x)1_{B_{r_\ell}}(y)]| \,dx dy\\&
	\leq 	\int_{{S}^2 \times {S}^2} 	| {E}[  \Psi_{\ell;\eps}(x,4)H_4(T_\ell(y))]| |  \varphi_\ell^i(x) -  1_{B_{r_\ell}}(x)| \varphi_\ell^i(y) \,dx dy\\&
	+	\int_{{S}^2 \times {S}^2} 	| {E}[  \Psi_{\ell;\eps}(x,4)H_4(T_\ell(y))]| |  \varphi_\ell^i(y) -  1_{B_{r_\ell}}(y)| 1_{B_{r_\ell}}(x) \,dx dy
	\end{split}
	\end{equation}
	and since we can bound  $	|{E}[\Psi_{\eps}(x)H_4(T_\ell(y))]| $ (see \cite{MRW}, Proposition 3.1), $ 1_{B_{r_\ell}}(x) $ and $ \varphi_\ell^i(y)$, (\ref{A4}) goes to zero as $i \rightarrow \infty$.
	
	Now, applying Fubini, equation (\ref{nonso}) is equal to
	\begin{equation}
	\begin{split}
	\lim_{i \rightarrow \infty}& -\dfrac{1}{4}\sqrt{\dfrac{\ell(\ell+1)}{2}}\dfrac{1}{4!} 8\pi^2 \int_{0}^{2r_\ell} {E}[ \Psi_{\ell;\eps}(x,4)H_4(T_\ell(y))] W^{\varphi_{\ell}^i}(\rho) \, d\rho.
	\end{split}
	\end{equation}
	In \cite{MRW}, Proposition 3.1, it is proved the term $ {E}[ \Psi_{\ell;\eps}(x,4)H_4(T_\ell(y))]$ can be computed explicitly and it is easily seen to be absolutely bounded for fixed $\ell$, uniformly over $\eps$. Hence, by the Dominated Convergence Theorem we may exchange the limit and the integral to obtain
	\begin{equation}\label{nS}
	\begin{split}
	&\Cov(\mathcal{Z}_{\ell,r_\ell}, \mathcal{M}_{\ell,r_\ell})= \lim_{\eps \rightarrow 0} \Cov(\mathcal{Z}_{\ell,r_\ell;\eps}, \mathcal{M}_{\ell,r_\ell})\\&
	=\lim_{\eps \rightarrow 0} \lim_{i \rightarrow \infty} -\dfrac{1}{4}\sqrt{\dfrac{\ell(\ell+1)}{2}} \frac{1}{4!}\times 8\pi^2 \int_{0}^{2r_{\ell}} {E}[\Psi_{\ell,\eps}(\bar{x},4) H_4(T_\ell(y(\rho)))] W^{\varphi_\ell^i}(\rho) d\rho\\&
	=\lim_{i \rightarrow \infty} -\dfrac{1}{4}\sqrt{\dfrac{\ell(\ell+1)}{2}} \frac{1}{4!}\times 8\pi^2 \int_{0}^{2r_{\ell}} \lim_{\eps \rightarrow 0} {E}[\Psi_{\ell,\eps}(\bar{x},4) H_4(T_\ell(y(\rho)))] W^{\varphi_\ell^i}(\rho) d\rho\\&=
	\lim_{i \rightarrow \infty} -\dfrac{1}{4}\sqrt{\dfrac{\ell(\ell+1)}{2}} \frac{1}{4!}\times 8\pi^2 \int_{0}^{2r_{\ell}}  {E}[\Psi_{\ell}(\bar{x},4) H_4(T_\ell(y(\rho)))] W^{\varphi_\ell^i}(\rho) d\rho\\&
	=\lim_{i \rightarrow \infty} \int_{0}^{2r_{\ell}\ell}  \mathcal{J}_\ell(\psi,4) W^{\varphi_\ell^i}(\frac{\psi}{L}) d\psi.
	\end{split}
	\end{equation}
\end{proof}




\subsection{Auxiliary function property}
Let us recall the following definitions:
\begin{equation}\label{WdefS}
W^{\varphi_\ell}(\rho):=\dfrac{1}{8\pi^2} \int_{d(x,y)=\rho} \varphi_\ell(x) \varphi_\ell(y) \,dx \,dy \mbox{ } \mbox{ } \mbox{ } x,y \in  {S}^2;
\end{equation}
\begin{equation}\label{W_1S}
\tilde{W_1} \bigg(\rho \dfrac{1}{r_\ell} \bigg):= \frac{1}{8\pi^2}\int_{d(x,y)=\frac{\rho}{r_\ell}}\tilde{\varphi_\ell}(r_\ell x) \tilde{\varphi_\ell}(r_\ell y) \,dx\,dy \mbox{ }\mbox{ }\mbox{ } x,y \in {R}^2.
\end{equation}
We give now the proof of the following lemma, which gives relation (3.6) of the main article.
\begin{lem}\label{lem} 
	Let $W_{r_{\ell}}(\cdot)$ and $ \tilde{W_{1}}(\cdot)$ as in (\ref{WdefS}) and (\ref{W_1S}), respectively; then,
	\begin{equation}\label{relazione Wr e tildeW1}
	W_{r_{\ell}}(\rho)=r_{\ell}^3 \tilde{W_{1}}\left(\rho \frac{1}{r_{\ell}}\right)(1+O(\rho^2))
	\end{equation}
	as $r_{\ell} \rightarrow 0 $ uniformly for $\rho \in [0,2r_\ell]$.
\end{lem}

\begin{proof}
	We set
	$D_{\rho}:=\{ x \in B_{r_{\ell}}: B_{\rho}(x) \subset B_{r_{\ell}}\};$
	then
	$$W_{r_{\ell}}(\rho)= \dfrac{1}{8\pi^2} \int_{D_{\rho}} \len\{ y \in B_{r_{\ell}} :d(x,y)=\rho \}\,dx+\dfrac{1}{8\pi^2} \int_{B_{r_\ell}-D_{\rho}} \len\{ y \in B_{r_{\ell}} :d(x,y)=\rho \}\,dx;$$
	we denote
	$$A:=  \dfrac{1}{8\pi^2} \int_{D_{\rho}} \len\{ y \in B_{r_{\ell}} :d(x,y)=\rho \}\,dx$$
	and $$B:=\dfrac{1}{8\pi^2} \int_{B_{r_\ell}-D_{\rho}} \len\{ y \in B_{r_{\ell}} :d(x,y)=\rho \}\,dx.$$
	$A$ is easily computed to be
	\begin{equation}\label{D}
	\begin{split}
	A&= \dfrac{1}{8\pi^2} \int_{D_{\rho}} \len\{ y \in B_{r_{\ell}} :d(x,y)=\rho \}\,dx
	=\dfrac{1}{8\pi^2} 2\pi \sin \rho |D_{\rho}|\\&=\dfrac{1}{8\pi^2} 2\pi \sin \rho \cdot 2\pi(1-\cos(r_\ell-\rho)).
	\end{split}
	\end{equation}
	Let us define also
	$\tilde{D}_{\rho/r_{\ell}}:= \{ x \in  \tilde{B}_1: \tilde{B}_{\rho/r_{\ell}}(x) \subset \tilde{B}_1 \};$
	likewise, we write
	$$	\tilde{W}_{1}\left(\rho\frac{1}{r_{\ell}}\right)= \tilde{A}+\tilde{B};$$
	where
	$$\tilde{A}:=  \dfrac{1}{8\pi^2} \int_{\tilde{D}_{\rho/r_{\ell}}} \len\left\{ y \in \tilde{B}_{1} :d(x,y)=\frac{\rho}{r_{\ell}} \right\}\,dx$$
	and $$\tilde{B}:=\dfrac{1}{8\pi^2} \int_{\tilde{B}_1-\tilde{D}_{\rho/r_{\ell}}} \len\left\{ y \in \tilde{B}_{1} :d(x,y)=\frac{\rho}{r_{\ell}} \right\}\,dx.$$
	Note that $$\tilde{A}=\dfrac{1}{8\pi^2} 2\pi \frac{\rho}{r_{\ell}} |\tilde{D}_{\rho/r_{\ell}}|= \dfrac{1}{8\pi^2} 2\pi \frac{\rho}{r_{\ell}}  \pi \bigg(1-\frac{\rho}{r_{\ell}}\bigg)^2;$$
	then, using the Taylor expansion of the sine and cosine as $r_{\ell} \rightarrow 0 $ (and so $\rho \rightarrow 0$), we get
	\begin{equation}
	\begin{split}
	A&=\dfrac{1}{8\pi^2} 2\pi \rho(1+O(\rho^2)) \pi \cdot (r_\ell-\rho)^2 (1+O(\rho)^2+O(r_\ell^2))\\&
	=\dfrac{1}{8\pi^2} 2\pi \frac{\rho}{r_{\ell}}\cdot \pi\bigg(1-\frac{\rho}{r_{\ell}}\bigg)^2 r_{\ell}^3 (1+O(\rho^2))(1+O(\rho)^2+O(r_\ell^2))\\&
	=r_{\ell}^3 \tilde{A}(1+O(\rho^2)+O(r_\ell^2)).
	\end{split}
	\end{equation}
	Now we prove that
	$$|B-\tilde{B}|\ll O(r_\ell^4+\rho^4)$$ and thus (\ref{relazione Wr e tildeW1}) follows. So,
	\begin{equation}\label{B-Btilde}
	\begin{split}
	&|B-\tilde{B}|\leq \\&
	\leq \bigg|\dfrac{1}{8\pi^2} \int_{B_{r_\ell}-D_{\rho}} \!\!\!\! \len\{ y \in B_{r_{\ell}} :d(x,y)=\rho \}\,dx- \dfrac{1}{8\pi^2} \int_{\tilde{B}_1-\tilde{D}_{\rho/r_{\ell}}} \!\!\!\!  \len\{ y \in \tilde{B}_{1} :d(x,y)=\frac{\rho}{r_{\ell}} \}\,dx\bigg|\\&
	=\bigg|\dfrac{1}{8\pi^2} \int_{B_{r_\ell}-D_{\rho}} \!\! \len\{ y \in B_{r_{\ell}} :d(x,y)=\rho \}\,dx- \dfrac{1}{8\pi^2} \int_{\tilde{B}_{r_\ell}-\tilde{D}_{\rho}} \! \len\{ y \in \tilde{B}_{r_\ell} :d(x,y)=\rho \}\,dx \bigg|,
	\end{split}
	\end{equation}
	where $\tilde{B}_{r_\ell} \subset {R}^2$ is the disc of radius $r_\ell$ and
	$ \tilde{D}_{\rho}:= \{ x \in  \tilde{B}_{r_\ell}: \tilde{B}_{\rho}(x) \subset \tilde{B}_{r_{\ell}} \};$ then (\ref{B-Btilde}) results to be 
	$$\ll 2\pi(1-\cos r_{\ell})-2\pi(1-\cos (r_\ell-\rho))-[\pi r_\ell^2-\pi( r_\ell-\rho)^2] \ll O(r_\ell^4)+O(\rho^4).$$
\end{proof}
As a consequence we can prove the following result.
\begin{cor} Let $W^{\varphi_{\ell}^i}(\cdot)$ and $\tilde{W}^{\tilde\varphi_\ell^i}(\cdot)$ defined as (\ref{WdefS}) and (\ref{W_1S}), respectively; $\varphi_{{\ell}}^i$ satisfies	\begin{equation}\label{propertiesS}
	\begin{split}
	& \varphi_{{\ell}}^i \rightarrow 1_{B_{r_\ell}} \mbox{ in }L^1({S}^2), \\
	& V(\varphi_{\ell}^i) \rightarrow V(1_{B_{r_{\ell}}}) \mbox{ and }\\
	& ||\varphi_{\ell}^i||_{\infty} \leq ||1_{B_{r_{\ell}}}||_{\infty};
	\end{split}
	\end{equation} and $\tilde{W}^{\tilde\varphi_\ell^i}:= \varphi_{{\ell}}^i \circ \exp.$ Then
	as $\ell \rightarrow \infty$,
	\begin{equation}\label{Rel1S}
	W^{\varphi_{\ell}^i}(\rho)=r_\ell^3 \tilde{W}^{\tilde\varphi_\ell^i}\left(\rho \frac{1}{r_\ell}\right)(1+O(\rho^2)),
	\end{equation}
	as $r_{\ell} \rightarrow 0 $ uniformly for $\rho \in [0,2r_\ell]$.
\end{cor}

\begin{proof}
	We have that 
	\begin{equation}\label{nome}
	\begin{split}
	|W^{\varphi_{\ell}^i}(\rho)&-r_\ell^3 \tilde{W}^{\tilde{\varphi}_\ell^i}(\rho \frac{1}{r_\ell})(1+O(\rho^2))| \leq |W^{{\varphi}_{\ell}^i}(\rho)-W_{r_\ell}(\rho)|\\&
	+|W_{r_\ell}(\rho)-r_\ell^3 \tilde{W}_1(\rho \frac{1}{r_\ell})(1+O(\rho^2))|\\&+|r_\ell^3 \tilde{W}_1(\rho \frac{1}{r_\ell})(1+O(\rho^2))-r_\ell^3 \tilde{W}^{\tilde{\varphi}_\ell^i}(\rho \frac{1}{r_\ell})(1+O(\rho^2))|
	\end{split}
	\end{equation}
	and the former and the latter quantities of (\ref{nome}) go to zero for the $L^1$ convergence of $\varphi_{\ell}^i \rightarrow 1_{B_{r_\ell}}$ and $\tilde{\varphi}_\ell^i \rightarrow 1_{\tilde{B}_1}$; in fact
	\begin{equation}
	\begin{split}
	|W^{\varphi_{\ell}^i}(\rho)-W_{r_\ell}(\rho)| &\leq \int_{{S}^2 \times {S}^2} |\varphi_{\ell}^i(x)\varphi_{\ell}^i(y)-1_{B_{r_\ell}}(x)1_{B_{r_\ell}}(y)|\,dxdy\\&
	\leq \int_{{S}^2 \times {S}^2}|\varphi_{\ell}^i(x)||\varphi_{\ell}^i(y)-1_{B_{r_\ell}}(y)|\,dxdy\\&+\int_{{S}^2 \times {S}^2}|1_{B_{r_\ell}}(y)||\varphi_{\ell}^i(x)-1_{B_{r_\ell}}(x)|\,dxdy \rightarrow 0.
	\end{split}
	\end{equation}
	and the conclusion of the lemma follows.
\end{proof}

\section{Further result}

\subsection{The second chaotic component}
In the lemma below, we show that the second chaotic component of the nodal length has lower order than the fourth one.
\begin{lem}\label{2chaos} The second component of the chaos expansion of $\mathcal{Z}_{\ell,r_\ell}$ is, as $\ell\rightarrow \infty,$
	$$Proj(\mathcal{Z}_{\ell,r_{\ell}}|C_2)=O(r_{\ell}^2).$$
\end{lem}

\begin{proof}
	Theorem 1.1 
	shows that $\Var(\mathcal{Z}_{\ell,r_\ell}) \sim \log r_\ell \ell$ and Proposition 2.3  
	shows that the orthogonal projection of $\mathcal{Z}_{\ell,r_\ell}$ along a well chosen vector in the fourth chaos is close to $\mathcal{Z}_{\ell,r_\ell}$ itself, up to a normalized error of $O\bigg( \dfrac{1}{\log r_\ell \ell} \bigg)$. Thus, the projection of $\mathcal{Z}_{\ell,r_\ell}$ onto any chaos of order different from four has variance 
	$$O\bigg( \frac{1}{\log (r_\ell \ell )} \times \Var(\mathcal{Z}_{\ell,r_\ell}) \bigg)=O(r_\ell^2).$$ 
\end{proof}

\end{document}